\documentclass[12pt]{article}

\usepackage{amsmath}
\usepackage{amssymb}
\usepackage{amsthm}
\usepackage{enumerate}
\usepackage{cases}
\usepackage{stmaryrd}
\usepackage{color}
\usepackage{hyperref}

\usepackage[top=1 in, bottom=1 in, left=1 in, right=1 in]{geometry}

\numberwithin{equation}{section} 
\newcommand\numberthis{\stepcounter{equation}\tag{\theequation}} % numbering the equation in align*

\newcounter{constantno}
\newcommand{\constantnumber}[1]{\refstepcounter{constantno}\label{#1}}

\newtheorem{lemma}{Lemma}[section]
\newtheorem{theorem}[lemma]{Theorem}

\theoremstyle{definition}
\newtheorem{definition}[lemma]{Definition}

\theoremstyle{remark}

\newtheorem{remark}[lemma]{Remark}

\newcommand{\ba}{\bar{\alpha}}  
\newcommand{\bb}{\bar{\beta}}

\newcommand{\bm}{\bar{\mu}}

\newcommand{\la}{\langle}
\newcommand{\ra}{\rangle}

\begin{document}

\title{Pseudo-Harmonic Maps From Complete Noncompact Pseudo-Hermitian Manifolds To Regular Balls\footnotetext{\textbf{Keywords}: Sub-Laplacian Comparison Theorem, Regular Ball, Pseudo-Harmonic Maps, Horizontal Gradient Estimate, Liouville Theorem, Existence Theorem}\footnotetext{\textbf{MSC 2010}: 58E20, 53C25, 32V05}}

\author{Tian Chong \and Yuxin Dong\footnote{Supported by NSFC grant No. 11771087 and LMNS, Fudan.} \and Yibin Ren\footnote{Corresponding author. Supported by NSFC grant No. 11801517} \and Wei Zhang}

\date{}

\maketitle

\begin{abstract}
In this paper, we give an estimate of sub-Laplacian of Riemannian distance functions in pseudo-Hermitian geometry which plays a similar role as Laplacian comparison theorem in Riemannian geometry, and deduce a prior horizontal gradient estimate of pseudo-harmonic maps from pseudo-Hermitian manifolds to regular balls of Riemannian manifolds. 
As an application, Liouville theorem is established under the conditions of nonnegative pseudo-Hermitian Ricci curvature and vanishing pseudo-Hermitian torsion. 
Moreover, we obtain the existence of pseudo-harmonic maps from complete noncompact pseudo-Hermitian manifolds to regular balls of Riemannian manifolds.
\end{abstract}

\section{Introduction}
Inspired by Eells-Sampson's theorem, one natural problem is to consider the existence of harmonic maps from complete noncompact Riemannian manifolds. 
% It is always affirmative under some convexity conditions. 
Usually some convexity conditions on the images will lead this existence (cf. \cite{ding1991harmonic,li1998heat,Li1991heat}).
% For example, target manifolds have nonpositive sectional curvature (cf. \cite{ding1991harmonic,Li1991heat}) and the images of initial maps are contained in regular balls (cf. \cite{li1998heat}).
Based on elliptic theory, some existence theorems have been studied for generalized harmonic maps (cf. \cite{chen2012exist,Ni1999hermitian}).

The pseudo-harmonic map is an analogue of the harmonic map in pseudo-Hermitian geometry. 
Let $(M, \theta)$ be a pseudo-Hermitian manifold of real dimension $2m+1$ and $(N, h)$ be a Riemannian manifold. The horizontal energy of a smooth map $f : M \to N$ is defined by
\begin{align}
E_H (f) = \int_M |d_b f|^2 \theta \wedge (d \theta)^m
\end{align}
where $d_b f$ is the horizontal part of $df$. 
The pseudo-harmonic map is a critical point of $E_H$.
Hence it locally satisfies the following Euler-Lagrange equation
\begin{align}
\tau_H^i (f) \overset{\Delta}{=} \Delta_b f^i + \sum_{j, k} \Gamma^i_{j k} (f) \langle d_b f^j, d_b f^k \rangle =0, \label{a-local}
\end{align}
where $\Gamma^i_{jk}$'s are Christoffel symbols of Levi-Civita connection in $(N, h)$.
% The sub-Laplacian is a subelliptic operator which enjoys similar local Sobolev theorems as elliptic theory.
Here $\Delta_b$ denotes the sub-Laplacian which is a subelliptic operator enjoying nice regularity as elliptic operators.
By heat flow method, the Eells-Sampson's type theorem also holds for pseudo-harmonic maps (cf. \cite{chang2013existence,ren2018pseudo}).
The Dirichlet problem of pseudo-harmonic maps to regular balls has also been solved by Jost-Xu (cf. \cite{jost1998subelliptic}).

This paper studies pseudo-harmonic maps from complete noncompact pseudo-Hermitian manifolds to regular balls.
% Laplacian comparison theorem is an important tool to explore global problems in Riemannian geometry.
In order to establish some local estimates, we need sub-Laplacian comparison theorem in pseudo-Hermitian manifolds.
Actually such kinds of theorems have been investigated for Sasakian manifolds in \cite{Agrachev2015bishop,baudoin2017comparison,chang2018gradient,lee2013bishop}.
% Its extension to Sasakian geometry has been studied in \cite{Agrachev2015bishop,baudoin2017comparison,chang2018gradient,lee2013bishop}.
However, up to now, there is no satisfactory comparison theorem for a pseudo-Hermitian manifold, which is not Sasakian.
For our purpose, we will give a new sub-Laplacian comparison theorem for a pseudo-Hermitian manifold. Note that the Riemannian distance associated with Webster metric has better regularity than the Carnot-Carath\'eodory distance, and its variational theory is well studied in Riemannian geometry.
By the index comparison theorem in Riemannian geometry, we can derive the following estimate of sub-Laplacian of Riemannian distance on pseudo-Hermitian manifolds. 
Let $B_R (x_0)$ be the Riemannian geodesic ball of radius $R$ centered at $x_0 \in M$.
\begin{theorem} \label{c-thm-sub}
Suppose $(M^{2m +1}, \theta)$ is a complete pseudo-Hermitian manifold. If for some $k, k_1 \geq 0$,
\begin{align*}
R_* \geq - k, \mbox{ and } |A|, | \mbox{div} A | \leq k_1 , \quad \mbox{ on } B_R (x_0),
\end{align*}
where $R_*$ is the pseudo-Hermitian Ricci curvature and $A$ is the pseudo-Hermitian torsion, \constantnumber{cst-2} 
then there exists $ C_{\ref*{cst-2}} = C_{\ref*{cst-2}} (m)$ such that
\begin{align}
\Delta_b r \leq C_{\ref*{cst-2}} \left(\frac{1}{r} + \sqrt{1 + k + k_1 + k_1^2}  \right), \quad \mbox{ on } B_R (x_0) \setminus Cut(x_0),
\end{align}
where $r$ is the Riemannian distance from $x_0$ and $Cut (x_0)$ is the cut locus of $x_0$.
\end{theorem}
The proof will be given in Section \ref{sec-comparison}.
Based on this sub-Laplacian comparison theorem, we will establish the following local prior horizontal gradient estimate of pseudo-harmonic maps by maximum principle.

\begin{theorem} \label{d-thm-estimate}
Suppose that $(M^{2m+1}, \theta)$ is a noncompact complete pseudo-Hermitian manifold and $(N, h)$ is a Riemannian manifold with sectional curvature $K^N \leq \kappa$ for some $\kappa \geq 0$. On $B_{2 R} (x_0) \subset M$ with $R >1$, 
\begin{align}
R_* \geq - k \quad \mbox{and} \quad |A|, | \mbox{div} A| \leq k_1,
\end{align}
for some $k, k_1 \geq 0$.
Assume that $f : B_{2 R} (x_0) \subset M \to B_D (p_0) \subset N $ is pseudo-harmonic where $B_D (p_0)$ is a regular ball in $N$. Then the horizontal energy density $|d_b f|^2$ on $B_R (x_0)$ is uniformly bounded. 
More precisely, 
\constantnumber{cst-reebcrbochner} \constantnumber{cst-3} 
\begin{align}
\max_{B_R (x_0)} |d_b f|^2 \leq C_{\ref*{cst-3}} \left[ C_{\ref*{cst-reebcrbochner}} + \frac{C_{\ref*{cst-reebcrbochner}}}{C_{\ref*{cst-reebcrbochner}} + R^{-1}} + \frac{1}{R} \right]
\end{align}
where $C_{\ref*{cst-reebcrbochner}}$ is given in Lemma \ref{b-lem-estimate} which depends on $k, k_1$ and $C_{\ref*{cst-3}}$ depends on $k, k_1, \kappa, D$.
In particular, if $k = 0$ and $k_1 = 0$, then $C_{\ref*{cst-reebcrbochner}} = 0$.
\end{theorem}
The proof will be given in Section \ref{sec-gradient}.
A direct application is the following Liouville theorem for pseudo-harmonic maps which is a generalization of the one for harmonic maps by Choi \cite{choi1982liouville}.
\begin{theorem} \label{d-thm-liouville}
Suppose that $(M, \theta)$ is a noncompact complete Sasakian manifold with nonnegative pseudo-Hermitian Ricci curvature and $(N, h)$ is a Riemannian manifold with sectional curvature bounded above. Then there is no nontrivial pseudo-Hermitian map from $M$ to any regular ball of $N$. 
\end{theorem}
% The similar Liouville theorem for harmonic case was deduced by Choi \cite{choi1982liouville}.
Another application of Theorem \ref{d-thm-estimate} is the global existence of pseudo-harmonic maps from complete noncompact pseudo-Hermitian manifolds to regular balls which is due to an exhaustion process combined with the Dirichlet existence of pseudo-harmonic maps.
% approach method by the solutions of Dirichlet problem on exhaustion domains.
% Hence the solutions of Dirichlet problem on exhaustion regions can approach a global pseudo-harmonic map.
\begin{theorem} \label{e-thm-exist}
Suppose that $(M, \theta)$ is a complete noncompact pseudo-Hermitian manifold and $(N, h)$ is a Riemannian manifold with sectional curvature bounded from above. Then there is a pseudo-harmonic map from $M$ to any regular ball $B_D (p_0)$ of $N$.
\end{theorem}
The proof will be given in Section \ref{sec-existence}.
One may doubt whether the pseudo-harmonic map given by Theorem \ref{e-thm-exist} is trivial. We will show an example whose domain is Sasakian with negative pseudo-Hermitian Ricci curvature.

\section{Basic Notions}

In this section, we present some basic notions of pseudo-Hermitian geometry and pseudo-harmonic maps. 
For details, readers may refer to \cite{dragomir2006differential,tanaka1975differential,webster1978pseudo}. Recall that a smooth manifold $M$ of real dimension $2m+1$ is said to be a CR manifold if
there exists a smooth rank $n$ complex subbundle $T_{1,0} M \subset TM \otimes \mathbb{C}$ such that
\begin{gather}
T_{1,0} M \cap T_{0,1} M = \{0\} \\
[\Gamma (T_{1,0} M), \Gamma (T_{1,0} M)] \subset \Gamma (T_{1,0} M) \label{b-integrable}
\end{gather}
where $T_{0,1} M = \overline{T_{1,0} M}$ is the complex conjugate of $T_{1,0} M$.
Equivalently, the CR structure may also be described by the real subbundle $HM = Re \: \{ T_{1,0}M \oplus T_{0,1}M \}$ of $TM$ which carries an almost complex structure $J : HM \rightarrow HM$ defined by $J (X+\overline{X})= i (X-\overline{X})$ for any $X \in T_{1,0} M$.
Since $HM$ is naturally oriented by the almost complex structure $J$, then $M$ is orientable if and only if
there exists a global nowhere vanishing 1-form $\theta$ such that $ HM = Ker (\theta) $.
Any such section $\theta$ is referred to as a pseudo-Hermitian structure on $M$.
The Levi form $L_\theta $ of a given pseudo-Hermitian structure $\theta$ is defined by
\[L_\theta (X, Y ) = d \theta (X, J Y) \quad  \mbox{ for any $X, Y \in HM$.} \]
An orientable CR manifold $(M, HM, J)$ is called strictly pseudo-convex if $L_\theta$ is positive definite for some $\theta$.
Such a quadruple $( M, HM, J, \theta )$ is called a pseudo-Hermitian manifold. 
For simplicity, we denote it by $(M, \theta)$.
% This paper is discussed in the pseudo-Hermitian manifolds.

For a pseudo-Hermitian manifold $(M, \theta)$, there exists a unique nowhere zero vector field $\xi$, called the Reeb vector field, transverse to $HM$ and satisfying
$\xi \lrcorner \: \theta =1, \ \xi \lrcorner \: d \theta =0$. 
It gives a decomposition of the tangent bundle $TM$: 
\begin{align}
TM = HM \oplus \mathbb{R} \xi 
\end{align}
which induces the projection $\pi_H : TM \to HM$. Set $G_\theta = \pi_H^* L_\theta$. Since $L_\theta$ is a metric on $HM$, it is natural to define a Riemannian metric
\begin{align}
g_\theta = G_\theta + \theta \otimes \theta
\end{align}
which makes $HM$ and $\mathbb{R} \xi$ orthogonal. 
The metric $g_\theta$ is called Webster metric, which is also denoted by $\la \cdot , \cdot \ra$ for simplicity.
By requiring $J \xi=0$, the almost complex structure $J$ can be extended to an endomorphism of $TM$.
% The integrable condition \eqref{b-integrable} guarantees that $g_\theta$ is $J$-invariant.
Clearly, $\theta \wedge (d \theta)^m$ differs a constant with the volume form of $g_\theta$.
Henceforth it is always regarded as the canonical volume form in pseudo-Hermitian geometry.

It is remarkable that $(M, HM, G_\theta)$ could also be viewed as a sub-Riemannian manifold which satisfies the strong bracket generating hypothesis (see Appendix for details). 
The completeness of a sub-Riemannian manifold is well settled under the Carnot-Carath\'eorody distance (cf. \cite{strichartz1986sub}). 
Locally, the Carnot-Carath\'eorody distance and the Riemannian distance associated with the Webster metric $g_\theta$ can be controlled by each other (cf. \cite{nagel1985balls}), which leads that the former completeness is equivalent with the latter.
% By definition, this distance is larger than Riemannian distance associated with the Webster metric $g_\theta$ which implies that sub-Riemannian completeness is stronger than Riemannian one. 
In this paper, a pseudo-Hermitian manifold $(M, \theta)$ is called complete if it is complete associated with the Webster metric $g_\theta$.

On a pseudo-Hermitian manifold, there exists a canonical connection $\nabla$, which is called Tanaka-Webster connection (cf. \cite{dragomir2006differential}), preserving the horizontal distribution, almost complete structure and Webster metric. Moreover, its torsion $T_\nabla$ satisfies
\begin{gather}
T_{\nabla} (X, Y)= 2 d \theta (X, Y) \xi \qquad \mbox{and } \quad T_{\nabla} (\xi, J X) + J T_{\nabla} (\xi, X) =0. \label{a-torsion}
\end{gather}
The pseudo-Hermitian torsion, denoted by $\tau$, is a symmetric and traceless tensor defined by $\tau (X) = T_\nabla (\xi, X)$ for any $X \in TM$ (cf. \cite{dragomir2006differential}).
Set
\begin{align*}
A (X, Y) = g_\theta ( \tau (X) , Y), \quad \mbox{for any $X, Y \in TM$}.
\end{align*}
A pseudo-Hermitian manifold is Sasakian if $\tau \equiv 0$.
Sasakian geometry plays important roles in K\"ahler geometry and Einstein metrics (cf. \cite{boyer2008sasakian}).

Suppose that $(M,\theta)$ is a pseudo-Hermitian manifold of real dimension $2m+1$. 
Let $R$ be the curvature tensor of the Tanaka-Webster connection. 
Set
\begin{align*}
R (X , Y, Z, W) = \langle R(Z, W) Y, X \rangle, \quad \mbox{ for any } X, Y, Z, W \in TM.
\end{align*}
Let $\{ \eta_\alpha\}_{\alpha=1}^m$ be a local unitary frame of $T_{1,0} M$ and $R_{ABCD}$ be the components of $R$ under the frame $\{ \eta_0 = \xi, \eta_\alpha, \eta_{\bar{\alpha}} \}$.
Webster \cite{webster1978pseudo} derived the first Bianchi identity, i.e. 
\begin{align*}
R_{\ba \beta \lambda \bm} = R_{\ba \lambda \beta \bm}.
\end{align*}
The other components of $R$ can be expressed by the pseudo-Hermitian torsion and its derivative.
For example,
\begin{align*}
R_{\bar{\alpha} \beta \lambda \mu} = 2 i (A_{\beta \mu} \delta_{\bar{\alpha} \lambda} - A_{\beta \lambda} \delta_{\bar{\alpha} \mu}), \quad R_{\bar{\alpha} \beta 0 \mu} = - A_{\beta \mu, \bar{\alpha}}, \quad R_{\bar{\alpha} \beta 0 \bar{\mu}} = A_{\bar{\alpha} \bar{\mu}, \beta}
\end{align*}
where $A_{\beta \mu, \bar{\alpha}}, A_{\bar{\alpha} \bar{\mu}, \beta}$ are the components of $\nabla A$.
Tanaka \cite{tanaka1975differential} defined the pseudo-Hermitian Ricci tensor $R_*$ by 
\begin{align}
R_* X = - i \sum_{\lambda=1}^m R(\eta_\lambda, \eta_{\bar{\lambda}}) JX \quad \mbox{ for  any } X \in T_{1,0} M.
\end{align}
The pseudo-Hermitian scalar curvature is given by
\begin{align}
s = \frac{1}{2} trace_{G_\theta} R_* .
\end{align}
In this paper, we will use Einstein summation convention when there is a repeated index. Denote $R_{\lambda \bar{\mu}} = R_{\bar{\alpha} \alpha \lambda \bar{\mu}}$.
Hence by the first Bianchi identity, $R_* \eta_\alpha = R_{\alpha \bar{\beta}} \eta_\beta$ and $s = R_{\alpha \bar{\alpha}}$.

Assume that $(N,h)$ is a Riemannian manifold. Let $\{ \sigma^i \}$ be an local orthonormal frame of $T^*N$. Denote the Levi-Civita connection and the Riemannian curvature of $(N,h)$ by $\nabla^N$ and $R^N$ respectively. Suppose that $f: M \rightarrow N$ is a smooth map. The pullback connection on the pullback bundle $f^*(TN)$ and the Tanaka-Webster connection induce a connection on $TM \otimes f^*(TN)$, also denoted by $\nabla$. 

\begin{definition}
A smooth map $f: M \rightarrow N$ is called pseudo-harmonic if the tensor field
\begin{align*}
\tau_H (f) \stackrel{\Delta}{=} \mbox{trace}_{G_\theta} \nabla_b d_b f \equiv 0,
\end{align*}
where $\nabla_b d_b f$ is the restriction of $\nabla df$ onto $HM \times HM$.
\end{definition}

Actually, pseudo-harmonic maps are the Dirichlet critical points of the horizontal energy (cf. \cite{barletta2001pseudoharmonic,dragomir2006differential})
\begin{align}
E_H (f) = \frac{1}{2} \int_M |d_b f|^2 \theta \wedge (d \theta)^m
\end{align}
where $d_b f$ is the horizontal restriction of $df$.
The sub-Laplacian $\Delta_b u$ of a smooth function $u$ is defined by
\begin{align}
\Delta_b u = \mbox{trace}_{G_\theta} \nabla_b d_b u, \label{b-2}
\end{align}
which is viewed as the special case of $\tau_H$ acting on functions.

\begin{lemma}[CR Bochner Formulas, cf. \cite{chang2013existence,greenleaf1985first,Ren201447}] \label{b-lemma-bochner}
For any smooth map $f: M \rightarrow N $, we have
\begin{align}
\frac{1}{2} \Delta_b |d_b f|^2=& |\nabla_b d_b f|^2 + \la \nabla_b \tau_H  (f), d_b f \ra + 4 i (f^i_{\ba} f^i_{0 \alpha } - f^i_\alpha f^i_{0  \ba} ) \nonumber \\
& + 2 R_{\alpha \bb} f^i_{\ba} f^i_{\beta} - 2i (m-2) (f^i_\alpha f^i_\beta A_{\ba \bb}- f^i_{\ba} f^i_{\bb} A_{\alpha \beta}  ) \nonumber \\
&+ 2 (f^i_{\ba} f^j_{\beta} f^k_{\bb} f^l_{\alpha} R^N_{ijkl} +f^i_{\alpha} f^j_{\beta} f^k_{\bb} f^l_{\ba} R^N_{ijkl}   ) \\
\frac{1}{2} \Delta_b |f_0|^2 =& |\nabla_b f_0|^2 + \la \nabla_\xi \tau_H (f) , f_0 \ra + 2 f^i_{0} f^j_{\alpha} f^k_{\ba} f^l_{0} R^N_{ijkl} \nonumber \\
& + 2 (f^i_0 f^i_\beta A_{\bb \ba, \alpha} + f^i_0 f^i_{\bb} A_{\beta \alpha, \ba} + f^i_0 f^i_{\bb \ba} A_{\beta \alpha } + f^i_0 f^i_{\beta \alpha } A_{\bb \ba} )
\end{align}
where $f^i_A$ and $f^i_{AB}$ are the components of $d f$ and $\nabla df$ respectively under the orthonormal coframe $ \{ \theta , \theta^\alpha, \theta^{\ba} \}$ of $T^*M$ and an orthonormal frame $\{ \sigma_i \}$ of $T^* N$, and $f_0 = df (\xi)$.
\end{lemma}

Let $\pi_{(1,1)} \nabla_b d_b f $ be the $(1,1)$-part of $\nabla_b d_b f$ and
\begin{align*}
\pi_{(1,1)}^\perp \nabla_b d_b f = \nabla_b d_b f - \pi_{(1,1)} \nabla_b d_b f
\end{align*}
which is orthogonal to $\pi_{(1,1)} \nabla_b d_b f$. The commutation relation (cf. \cite{chang2013existence,Ren201447})
\begin{align}
f^i_{\alpha \bar{\beta}} - f^i_{\bar{\beta} \alpha} = 2 i f^i_0 \delta_{\alpha \bar{\beta}} \label{b-commutation-1}
\end{align}
shows that 
\begin{align}
| \pi_{(1,1)} \nabla_b d_b f|^2 \geq & 2 \sum_{\alpha = 1}^m f^i_{\alpha \bar{\alpha}} f^i_{\bar{\alpha} \alpha}  \nonumber \\
= & \frac{1}{2} \sum_{\alpha=1}^m \big[ |f^i_{\alpha \ba} +f^i_{\ba \alpha} |^2 + |f^i_{\alpha \ba} - f^i_{\ba \alpha} |^2  \big] \nonumber \\ 
\geq &\frac{1}{2} \sum_{\alpha=1}^m |f^i_{\alpha \ba} -f^i_{\ba \alpha} |^2  \nonumber \\
=& 2m |f_0|^2. \label{b-commutation}
\end{align}
Combining with Lemma \ref{b-lemma-bochner}, we have the following lemma.

% \constantnumber{cst-reebcrbochner}
\begin{lemma} \label{b-lem-estimate}
Suppose that $(M^{2m+1}, \theta)$ is a pseudo-Hermitian manifold with
\begin{align}
R_* \geq -k, \mbox{ and } |A|, |\mbox{div } A| \leq k_1
\end{align}
and $(N, h)$ is a Riemannian manifold with sectional curvature 
\begin{align}
K^N \leq \kappa
\end{align}
for $k, k_1, \kappa \geq 0$. Then there exists $C_{\ref*{cst-reebcrbochner}} = C_{\ref*{cst-reebcrbochner}} (k, k_1)$ such that for any pseudo-harmonic map $f : M \to N$, we have
\begin{align}
\Delta_b |d_b f|^2 \geq & (2- \epsilon) |\nabla_b d_b f|^2 + 2m \epsilon |f_0|^2 + \epsilon |\pi_{(1,1)}^{\perp} \nabla_b d_b f|^2 \nonumber \\
& - \epsilon_1 |\nabla_b f_0|^2 - (C_{\ref*{cst-reebcrbochner}} + 16 \epsilon_1^{-1} ) |d_b f|^2 -2 \kappa |d_b f|^4 \label{b-bochner1}
\end{align}
and
\begin{align}
\Delta_b |f_0|^2 \geq  2 |\nabla_b f_0|^2 -2 \kappa |f_0|^2 |d_b f|^2 - C_{\ref*{cst-reebcrbochner}} | \pi_{(1,1)}^{\perp} \nabla_b d_b f |^2 - C_{\ref*{cst-reebcrbochner}} |f_0|^2 - C_{\ref*{cst-reebcrbochner}} |d_b f|^2 \label{b-bochner2}
\end{align}
where $\epsilon$ and $\epsilon_1$ are any positive number.
In particular, if $k =0$ and $k_1 =0$, then $C_{\ref*{cst-reebcrbochner}} = 0$.
\end{lemma}

\begin{proof}
For \eqref{b-bochner1}, due to \eqref{b-commutation}, Cauchy inequality and the identity
\begin{align*}
i (f^i_{\bar{\alpha}} f^i_{0 \alpha} - f^i_\alpha f^i_{0 \bar{\alpha}}) = -\langle \nabla_b f_0 , d_b f \circ J \rangle,
\end{align*}
it suffice to prove that
\begin{align}
f^i_{\ba} f^j_{\beta} f^k_{\bb} f^l_{\alpha} R^N_{ijkl} +f^i_{\alpha} f^j_{\beta} f^k_{\bb} f^l_{\ba} R^N_{ijkl} \geq - \frac{1}{2} \kappa |d_b f|^4. \label{b-1}
\end{align}
Set
\begin{align*}
d f (\eta_\alpha) = t_\alpha + i t_\alpha' .
\end{align*}
Hence due to sectional curvature $K^N \leq \kappa$, a direct calculation shows that 
\begin{align*}
& f^i_{\ba} f^j_{\beta} f^k_{\bb} f^l_{\alpha} R^N_{ijkl} +f^i_{\alpha} f^j_{\beta} f^k_{\bb} f^l_{\ba} R^N_{ijkl} \\
& = 2 \big(\langle R^N (t_\beta, t_\alpha) t_\beta, t_\alpha \rangle + \langle R^N (t_\beta, t_\alpha') t_\beta, t_\alpha' \rangle + \langle R^N (t_\beta', t_\alpha) t_\beta', t_\alpha \rangle + \langle R^N (t_\beta', t_\alpha') t_\beta', t_\alpha' \rangle \big) \\
& \geq - 2 \kappa \sum_{\alpha, \beta = 1}^m (|t_\alpha|^2 |t_\beta|^2 + |t_\alpha'|^2 |t_\beta|^2 + |t_\alpha|^2 |t_\beta'|^2 + |t_\alpha'|^2 |t_\beta'|^2) \\
& = -2 \kappa \left( \sum_{\alpha =1}^m (|t_\alpha|^2 + |t_\alpha'|^2 ) \right) \left( \sum_{\beta =1}^m (|t_\beta|^2 + |t_\beta'|^2 ) \right)
\end{align*}
which, combining with 
\begin{align*}
|d_b f |^2 = 2 \sum_{\alpha=1}^m \langle d f(\eta_\alpha), df (\eta_{\bar{\alpha}}) \rangle = 2 \sum_{\alpha=1}^m \langle t_\alpha + i t_\alpha', t_\alpha - i t_\alpha' \rangle = 2 \sum_{\alpha=1}^m  (|t_\alpha|^2 + |t_\alpha'|^2 ) 
\end{align*}
yields \eqref{b-1}.

Similarly, \eqref{b-bochner2} follows from the following process
\begin{align*}
f^i_{0} f^j_{\alpha} f^k_{\ba} f^l_{0} R^N_{ijkl} & = \langle R^N (t_\alpha - i t_\alpha', f_0) (t_\alpha + i t_\alpha'), f_0 \rangle \\
& = \langle R^N (t_\alpha, f_0) t_\alpha, f_0 \rangle + \langle R^N (t_\alpha' , f_0) t_\alpha', f_0 \rangle \\
& \geq - \kappa |f_0|^2 \left( \sum_{\alpha =1}^m (|t_\alpha|^2 + |t_\alpha'|^2 ) \right) \\
& = - \frac{1}{2} \kappa |f_0|^2 |d_b f|^2.
\end{align*}
\end{proof}

At the end of this Section, we briefly recall Folland-Stein space. Let $(M,\theta)$ be a pseudo-Hermitian manifold and $\Omega \Subset M$. For any $k \in \mathbb{N}$ and $p>1$, the Folland-Stein space $S^p_k (\Omega)$ is given by 
\begin{align*}
S_k^p (\Omega) = \big\{ u \in L^p (\Omega) \big| \: \nabla_b^l u \in L^p (\Omega), l = 0, 1, \dots, k \big\}
\end{align*}
where $\nabla^l_b u$ is the horizontal restriction of $\nabla^l u$ and its $S^p_k$-norm is defined by 
\begin{align}
||u||_{S^p_k (\Omega)} = \sum_{l=0}^k ||\nabla_b^l u||_{L^p (\Omega)}, \label{b-sobolev-1}
\end{align}
which is equivalent to the local Folland-Stein norm in \cite{dragomir2006differential} (see Appendix for details).
Under this generalized Sobolev space, the interior regularity theorem of subelliptic equations will behave as elliptic ones.

\begin{theorem} \label{b-lem-regular}
Suppose that $(M,\theta)$ is a pseudo-Hermitian manifold and $\Omega \Subset M$. Assume that $u, v \in L_{loc}^1 (\Omega)$ and $\Delta_b u = v$ in the distribution sense. For any $\chi \in C^\infty_0 (\Omega)$, if $v \in S^p_k (\Omega) $ with $p >1$ and $k \in \mathbb{N}$, then $\chi u \in S^p_{k+2} (\Omega)$ and 
\begin{align}
||\chi u||_{S^p_{k+2} (\Omega)} \leq C_{\chi} \left( ||u||_{L^p (\Omega)} + ||v||_{S^p_k (\Omega)} \right)
\end{align}
where $C_{\chi}$ only depends on $\chi$.
\end{theorem}

The proof is based on partition of unity and the corresponding version on coordinate neighborhoods (cf. Theorem 3.17 in \cite{dragomir2006differential}, Theorem 16 in \cite{rothschild1976hypoelliptic}). For completeness, we will give the details in Appendix.
A direct calculation shows that for any $ \sigma \in \Gamma ( \otimes^k T^* M )  $ and $ X_1 , \cdots , X_k, X, Y \in \Gamma (HM) $, we have
\begin{align*}
	& ( \nabla^2 \sigma )   ( X_1 , \cdots , X_k ; X, Y ) - ( \nabla^2 \sigma )   ( X_1 , \cdots , X_k ; Y, X ) \nonumber \\
	&= \sum_{i=1}^k \sigma  ( X_1 , \cdots , R(X, Y) X_i, \cdots,  X_k ) + \big( \nabla_{T_\nabla (X, Y) }  \sigma \big) ( X_1 , \cdots , X_k ).
\end{align*}
By taking $\sigma = \nabla_b^k u $ and $X = \eta_\alpha, Y= \eta_{\bar{\beta}}$, we obtain that
\begin{align*}
2 i \delta_{\alpha \bar{\beta}} \nabla_\xi \nabla_b^k u (X_1, \cdots, X_k) 
& = ( \nabla_b^{k+2} u )   ( X_1 , \cdots , X_k ; \eta_\alpha, \eta_{\bar{\beta}} ) - ( \nabla_b^{k+2} u )   ( X_1 , \cdots , X_k ; \eta_{\bar{\beta}}, \eta_\alpha ) \\
& \quad - \sum_{i=1}^k \nabla_b^k u  \big( X_1 , \cdots , R(\eta_\alpha, \eta_{\bar{\beta}}) X_i, \cdots,  X_k \big)
\end{align*}
which implies that Reeb covariant derivatives can be controlled by horizontal covariant derivatives.
Hence the Folland-Stein space may be embedded into some classical Sobolev space which is a generalization of Theorem 19.1 in \cite{folland1974estimates}.

\begin{theorem}
Suppose that $(M, \theta)$ is a pseudo-Hermitian manifold and $\Omega \Subset M$. 
Then for any $k \in \mathbb{N}$ and $p >1$, $$S^p_k (\Omega) \subset L^p_{k/2} (\Omega)$$ 
where $L^p_{k/2} (\Omega)$ is the classical Sobolev space. Moreover, for any $r \in \mathbb{N}$ and $p > \mbox{dim } M$, there exists $k \in \mathbb{N}$ such that
\begin{align*}
S^p_k (\Omega) \subset C^{r, \alpha} (\Omega).
\end{align*}
\end{theorem}

\section{Sub-Laplacian Comparison Theorem} \label{sec-comparison}
In this section, we will deduce Theorem \ref{c-thm-sub} which plays a similar role as Laplacian comparison theorem in Riemannian geometry.
% This section will estimate the lower bound of sub-Laplacian of Riemannian distance function under some conditions of pseudo-Hermitian Ricci curvature and pseudo-Hermitian torsion, which plays a similar role as Laplacian comparison theorem in Riemannian geometry.

Suppose that $(M^{2m+1}, \theta)$ is a complete noncompact pseudo-Hermitian manifold. 
Let $r$ be the Riemannian distance with respect to Webster metric $g_\theta$ from a reference point $x_0 \in M$.
We formulate all Riemannian symbols with ``$\hat{\phantom{a}}$" to distinguish with ones in pseudo-Hermitian geometry, such as Levi-Civita connection $\hat{\nabla}$ and Riemannian curvature tensor $\hat{R}$.
Lemma 1.3 in \cite{dragomir2006differential} shows the relation of Tanaka-Webster connection and Levi-Civita connection associated with Webster metric:
\begin{align}
\hat{\nabla} = \nabla - (d \theta + A) \otimes \xi + \tau \otimes \theta + 2 \theta \odot J \label{c-connection}
\end{align}
where $2 \theta \odot J = \theta \otimes J + J \otimes \theta$. 
Hence the sub-Laplacian of $r$ can also be calculated by Levi-Civita connection as follows:
\begin{align}
\Delta_b r = \mbox{trace}_{G_\theta} \widehat{\mbox{Hess}} (r) \big|_{HM \times HM}
\end{align}
where $\widehat{Hess}$ is the Riemannian Hessian.

Let's recall the Index Lemma in Riemannian geometry (cf. \cite{do1992riemannian} in page 212).

\begin{lemma}[Index Lemma]
Let $\gamma : [0,a] \to M$ be a Riemannian geodesic without conjugate points to $\gamma (0)$ in $(0, a]$ and $X$ be a Jacobi field along $\gamma$ with $X \perp \dot{\gamma} $ and $X(0) =0$. If $V \in \Gamma (TM) \big|_{\gamma} $ with $V(0) =0, V(a) = X(a)$ and $V \perp \dot{\gamma}$. Then
\begin{align}
I_a (X, X) \leq I_a (V, V)
\end{align}
where
\begin{align*}
I_a (V, V) = \int_0^a \left( \big| \hat{\nabla}_{\dot{\gamma}} V \big|^2 - \langle \hat{R} (V, \dot{\gamma}) \dot{\gamma}, V \rangle \right) dt
\end{align*}
\end{lemma}

Now let $\gamma : [0,a] \to M$ be such a geodesic and $\{ e_B (a) \}_{B=1}^{2m} $ be an orthonormal basis of $HM \big|_{\gamma(a)} $. Set 
\begin{align*}
e_B^{\perp} (a) = e_B (a) - \langle e_B (a), \nabla r  \rangle \nabla r  \in TM \big|_{\gamma (a)}
\end{align*}
which is perpendicular to $\dot{\gamma} (a) = \nabla r \big|_{\gamma (a)}$.
Since $\widehat{Hess} (r) (\nabla r, \cdot) =0 $, then 
\begin{align}
\Delta_b r \big|_{\gamma (a)} = \sum_{B=1}^{2m} \widehat{Hess} (r) (e_B (a), e_B (a)) = \sum_{B=1}^{2m} \widehat{Hess} (r) (e_B^{\perp} (a), e_B^\perp (a))
\end{align}
Using the Riemannian exponential map, we could extend $e_B^{\perp} (a) $ as a Jacobi field $U_B$ along $\gamma$ with
\begin{align*}
U_B (0) = 0, U_B (a) = e_B^{\perp} (a), [U_B , \dot{\gamma}] = 0.
\end{align*}
Hence we find
\begin{align*}
\widehat{Hess} (r) (e_B^{\perp} (a), e_B^\perp (a)) & = \widehat{Hess} (r) ( U_B (a), U_B (a))  \\
& = \langle U_B  , \hat{\nabla}_{U_B} \nabla r \rangle \big|_{\gamma (a)} = \langle U_B  , \hat{\nabla}_{\dot{\gamma}} U_B \rangle \big|_{\gamma (a)} = \int_0^a \frac{d}{dt} \langle U_B , \hat{\nabla}_{\dot{\gamma}} U_B \rangle dt\\
&  = \int_0^a \left( \big| \hat{\nabla}_{\dot{\gamma}} U_B \big|^2 + \langle U_B, \hat{\nabla}_{\dot{\gamma}} \hat{\nabla}_{\dot{\gamma}} U_B \rangle \right)  dt = I_a (U_B, U_B),
\end{align*}
where the last equation is due to the Jacobi equation.
Hence 
% the sub-Laplacian of $r$ can be calculated by
\begin{align}
\Delta_b r \big|_{\gamma (a)} = \sum_{B =1}^{2m } I_a (U_B, U_B). \label{c-subhess}
\end{align}

\begin{lemma} \label{c-lem1}
Let $e_B (t) $ be the parallel extension of $e_B (a)$ along $\gamma$ with respect to Tanaka-Webster connection.
Suppose the curvature along $\gamma$ satisfies
\begin{align}
\sum_{B=1}^{2m} \langle \hat{R} (e_B, \nabla r ) \nabla r , e_B \rangle \geq - \hat{k} \label{c-lem1cur}
\end{align}
and the pseudo-Hermitian torsion is bounded, i.e.
\begin{align}
|A| \leq k_1,
\end{align}
for some for $\hat{k}, k_1 \geq 0$.
\constantnumber{cst-1}
Then there is a constant $C_{\ref*{cst-1}} = C_{\ref*{cst-1}} (m)$ such that
\begin{align}
\Delta_b r \big|_{\gamma (a)} \leq C_{\ref*{cst-1}} \left(\frac{1}{a} + \sqrt{1 + k_1 + k_1^2 + \hat{k}} \right).
\end{align}
\end{lemma}

\begin{proof}
Due to \eqref{c-connection}, we have
\begin{align*}
\hat{\nabla}_{\dot{\gamma}} e_B = - [ d \theta (\dot{\gamma}, e_B) + A (\dot{\gamma}, e_B) ] \xi + \theta(\dot{\gamma}) J e_B = - [ g_\theta (J \dot{\gamma}, e_B) + A (\dot{\gamma}, e_B) ] \xi + \theta(\dot{\gamma}) J e_B
\end{align*}
which implies that
\begin{align*}
\sum_{B=1}^{2m} \left| \hat{\nabla}_{\dot{\gamma}} e_B \right|^2 
& = 2m \left| \theta (\dot{\gamma}) \right|^2 + \sum_{B=1}^{2m} \left[ \left| g_\theta (J \dot{\gamma}, e_B) \right|^2 + 2 g_\theta (J \dot{\gamma}, e_B) A (\dot{\gamma}, e_B) + \left| A (\dot{\gamma}, e_B) \right|^2 \right] \\
& \leq 2m + 2 A (\dot{\gamma}, J \dot{\gamma}) + \sum_{B=1}^m \left| A (\dot{\gamma}, e_B) \right|^2 \leq 2m + 2 k_1 + k_1^2 .
\end{align*}
Set
\begin{gather*}
e_B' (t) = e_B (t) - \langle e_B (t), \nabla r \rangle \nabla r \perp \dot{\gamma} , \qquad
V_B (t) = \frac{s_{\kappa} (t) }{s_{\kappa} (a)} e_B' (t),
\end{gather*}
where
\begin{align*}
s_{\kappa} (t) = \frac{1}{\sqrt{\kappa}} \sinh (\sqrt{\kappa} t)  \qquad \mbox{and} \quad \kappa = \frac{1}{4m} (4m + 4 k_1 + 2 k_1^2 + \hat{k}).
\end{align*}
Hence $V_B (0) = 0$, $V_B (a) = e_B' (a), V_B \perp \dot{\gamma}$ and
\begin{align*}
\sum_{B=1}^{2m} \left| \hat{\nabla}_{\dot{\gamma}} V_B \right|^2 
& = \sum_{B=1}^{2m} \left| \frac{\dot{s}_{\kappa} (t) }{s_{\kappa} (a) } e_B' + \frac{s_{\kappa} (t) }{s_{\kappa} (a) } \hat{\nabla}_{\dot{\gamma}} e_B' \right|^2 \leq \frac{3}{2} \sum_{B=1}^{2m} \left| \frac{\dot{s}_{\kappa} (t) }{s_{\kappa} (a) } e_B' \right|^2 + 3 \sum_{B=1}^{2m} \left| \frac{s_{\kappa} (t) }{s_{\kappa} (a) } \hat{\nabla}_{\dot{\gamma}} e_B' \right|^2 \\
& \leq 4 m \left| \frac{\dot{s}_{\kappa} (t) }{s_{\kappa} (a) } \right|^2 + (4m + 4 k_1 + 2 k_1^2) \left| \frac{s_{\kappa} (t) }{s_{\kappa} (a) } \right|^2
\end{align*}
due to Cauchy inequality. 
By the curvature assumption, the Index lemma and \eqref{c-subhess}, we have
\begin{align*}
\Delta_b r \big|_{\gamma (a)} & \leq \sum_{B=1}^{2m} I_a (V_B, V_B) = \sum_{B=1}^{2m} \int_0^a \left( \big| \hat{\nabla}_{\dot{\gamma}} V_B \big|^2 - \langle \hat{R} (V_B, \nabla r) \nabla r , V_B \rangle \right) dt \\
& = \int_0^a \left( 4 m \left| \frac{\dot{s}_{\kappa} (t) }{s_{\kappa} (a) } \right|^2 + (4m  + 4 k_1 + 2 k_1^2 + \hat{k}) \left| \frac{s_{\kappa} (t) }{s_{\kappa} (a) } \right|^2 \right) dt \\
& \leq \frac{4m}{| s_{\kappa} (a) |^2} \int_0^a \left( | \dot{s}_{\kappa} (t) |^2 + \kappa | s_{\kappa} (t) |^2 \right) dt \\
& = 4 m \sqrt{\kappa} \: \coth \sqrt{\kappa} a \\
& \leq 4 m ( \frac{1}{a} + \sqrt{\kappa} )
\end{align*}
which finishes the proof.
\end{proof}

To prove Theorem \ref{c-thm-sub}, it suffices to demonstrate \eqref{c-lem1cur}.
It can be expressed by pseudo-Hermitian data due to the relationship between the Riemannian curvature tensor $\hat{R}$ and the curvature tensor $R$ associated with Tanaka-Webster connection $\nabla$ (cf. Theorem 1.6 in \cite{dragomir2006differential}):
\begin{align*}
\hat{R} (X, Y) Z = & R (X, Y) Z + ( LX \wedge LY ) Z  + 2 d \theta (X, Y) J Z  \\
& \quad  - g_\theta ( S (X, Y), Z ) \xi + \theta (Z ) S ( X, Y)  \\
& \quad - 2 g_\theta ( \theta \wedge \mathcal{O} (X, Y), Z) \xi  + 2 \theta ( Z) (\theta \wedge \mathcal{O}) ( X, Y ) \numberthis \label{c-2}
\end{align*}
where
\begin{align*}
S (X, Y) =& ( \nabla_X \tau ) Y - (\nabla_Y \tau) X \\
\mathcal{O} = & \tau^2 + 2 J \tau  - I \\
L = & \tau + J 
\end{align*}
Here $I$ is the identity, that is $I(X) = X$.
Note that the left side of \eqref{c-lem1cur} is independent of the choice of horizontal orthonormal frame of $\{e_B\}_{B=1}^{2m}$.
Let $\{e_B\}_{B=1}^{2m}$ be a local real orthonormal basis of $HM$ with $e_{\alpha+m} = J e_\alpha$ for $\alpha = 1, \dots m$. Denote $\eta_\alpha = \frac{1}{\sqrt{2}} (e_\alpha - i J e_\alpha)$. 

\begin{lemma} \label{c-riemric1}
For $X, Y \in TM$, we have
\begin{align}
\sum_{B=1}^{2m} \langle \hat{R} (e_B, X) Y,  e_B \rangle & = \sum_{B=1}^{2m} \langle R(e_B, X) Y, e_B \rangle - 3 \langle \pi_H X, \pi_H Y \rangle \notag \\
& \quad + \langle \tau X, \tau Y \rangle  + (2m - |\tau|^2) \theta (X) \theta (Y) + \mbox{\normalfont div } \tau (X) \theta (Y) \label{c-ric}
\end{align}
\end{lemma}

\begin{proof}
By \eqref{c-2} and $e_B \in HM$, we have 
\begin{align*}
& \sum_{B=1}^{2m} \langle \hat{R}(e_B, X) Y, e_B \rangle \\
& =  \sum_{B=1}^{2m} \langle R(e_B, X) Y, e_B \rangle + \sum_{B=1}^{2m} \langle (L e_B \wedge L X) Y, e_B \rangle + \sum_{B=1}^{2m} 2 d \theta (e_B, X) \langle JY, e_B \rangle \\
& \quad + \sum_{B=1}^{2m} \theta(Y) \langle S(e_B,X) , e_B \rangle + \sum_{B=1}^{2m} 2 \theta(Y) \langle (\theta \wedge \mathcal{O}) (e_B, X), e_B \rangle \numberthis \label{c-ric1}
\end{align*}
Now we see each terms in the right side except the first one. Note that
\begin{align}
\sum_{B=1}^{2m} \langle (L e_B \wedge L X) Y, e_B \rangle = \sum_{B=1}^{2m} \langle L e_B , Y \rangle  \langle LX, e_B \rangle - \langle LX, Y \rangle \langle L e_B, e_B \rangle \label{c-4}
\end{align}
On one hand, since $LX$ is horizontal and 
\begin{align*}
\langle L e_B, Y \rangle  = \langle e_B, \tau Y \rangle  - \langle e_B, J Y \rangle,
\end{align*}
then we find
\begin{align*}
\sum_{B=1}^{2m} \langle L e_B , Y \rangle  \langle LX, e_B \rangle =& \langle LX , \tau Y \rangle - \langle LX, JY \rangle \\
= & \langle \tau X, \tau Y \rangle + \langle JX, \tau Y \rangle - \langle \tau X, J Y \rangle - \langle J X, JY \rangle \\
= & \langle \tau X, \tau Y \rangle - \langle \pi_H X, \pi_H Y \rangle. \numberthis \label{c-6}
\end{align*}
Here the last equation is due to $\tau J + J \tau =0$ by \eqref{a-torsion}.
On the other hand, 
\begin{align}
\langle L e_B, e_B \rangle = trace_{G_\theta} \tau + trace_{G_\theta} J = 0. \label{c-7}
\end{align}
Substituting \eqref{c-6} and \eqref{c-7} into \eqref{c-4}, the result is 
\begin{align}
\sum_{B=1}^{2m} \langle (L e_B \wedge L X) Y, e_B \rangle = \langle \tau X, \tau Y \rangle - \langle \pi_H X, \pi_H Y \rangle. \label{c-8}
\end{align}
For the third term in \eqref{c-ric1}, we have
\begin{align}
\sum_{B=1}^{2m} 2 d \theta (e_B, X) \langle JY, e_B \rangle = \sum_{B=1}^{2m} 2 \langle J e_B, X \rangle \langle JY, e_B \rangle  = -2 \langle \pi_H X, \pi_H Y \rangle. \label{c-9}
\end{align}
For the fourth term in \eqref{c-ric1}, by the formula of $S$, we have 
\begin{align}
\sum_{B=1}^{2m} \langle S(e_B, X), e_B \rangle = \sum_{B=1}^{2m} \langle (\nabla_{e_B} \tau) X, e_B \rangle - \sum_{B=1}^{2m} \langle (\nabla_X \tau) e_B, e_B \rangle = \mbox{div } \tau (X) \label{c-10}
\end{align}
since $\tau$ is traceless.
For the fifth term, by the definition of $\mathcal{O}$, we have 
\begin{align}
\sum_{B=1}^{2m} 2 \langle (\theta \wedge \mathcal{O}) (e_B, X), e_B \rangle = & \sum_{B=1}^{2m} - \langle \theta (X) \mathcal{O} (e_B) , e_B \rangle \nonumber \\ 
= & \sum_{B=1}^{2m} -\theta (X) \langle (\tau^2 + 2 J \tau - I ) (e_B) , e_B \rangle \nonumber \\
= & \theta (X) (2m -|\tau|^2) \label{c-11}
\end{align}
due to 
\begin{align*}
- \sum_{B=1}^{2m} \langle J \tau (e_B), e_B \rangle = \sum_{B=1}^{2m} \langle \tau J e_B , e_B \rangle = \sum_{\alpha=1}^{m} \langle \tau J e_\alpha, e_\alpha \rangle + \langle \tau J^2 e_\alpha, J e_\alpha \rangle = 0.
\end{align*}
By substituting \eqref{c-8}, \eqref{c-9}, \eqref{c-10} and \eqref{c-11} to \eqref{c-ric1}, we get \eqref{c-ric}.
\end{proof}

Tanaka \cite{tanaka1975differential} obtained the following version of first Bianchi identity of $R$:
\begin{align}
\mathcal{S} \left( R(X, Y) Z \right) = 2 \mathcal{S} \left( d \theta (X, Y) \tau (Z) \right).  \label{c-firstbianchi}
\end{align}
where $\mathcal{S}$ stands for the cyclic sum with respect to $X, Y, Z \in HM$. 
One can prove it by applying Riemannian first Bianchi identity to \eqref{c-2}.

\begin{lemma} \label{c-ricrelation}
For any $X, Y \in TM$, we have
\begin{align}
\langle R_* X, Y \rangle =& \sum_{B=1}^{2m} \langle R (e_B, \pi_H X) \pi_H Y, e_B \rangle - 2 (m-1) A( X, J Y), \label{c-pseric} 
\end{align}
\end{lemma}

\begin{proof}
Since $J X$ is horizontal, we can use the first Bianchi identity \eqref{c-firstbianchi} and obtain
\begin{align*} 
-i & \sum_{\alpha=1}^m R(\eta_\alpha, \eta_{\bar{\alpha}}) JX - i \sum_{\alpha=1}^m  R(\eta_{\bar{\alpha}},  JX) \eta_\alpha - i \sum_{\alpha=1}^m R(JX, \eta_\alpha) \eta_{\bar{\alpha}} \\
&= - i \sum_{\alpha=1}^m 2 d \theta (\eta_\alpha, \eta_{\bar{\alpha}}) \tau J X - i \sum_{\alpha=1}^m 2 d \theta (\eta_{\bar{\alpha}}, JX) \tau \eta_\alpha - i \sum_{\alpha=1}^m 2 d \theta (JX, \eta_\alpha) \tau \eta_{\bar{\alpha}}  \\
& = 2m \tau JX - 2 \sum_{\alpha=1}^m \tau J \bigg( \langle \eta_{\bar{\alpha}}, X \rangle \eta_\alpha + \langle \eta_\alpha, X \rangle \eta_{\bar{\alpha}} \bigg) \\
& = 2 (m-1) \tau JX. \numberthis \label{c-3}
\end{align*}
On the other hand, note that 
\begin{align*}
i \sum_{\alpha=1}^m  R(\eta_{\bar{\alpha}},  JX) \eta_\alpha + i \sum_{\alpha=1}^m R(JX, \eta_\alpha) \eta_{\bar{\alpha}} = & - i \sum_{\alpha=1}^m  R(JX, \eta_{\bar{\alpha}}) \eta_\alpha + i \sum_{\alpha=1}^m R(JX, \eta_\alpha) \eta_{\bar{\alpha}}  \\
= & - J \left( \sum_{\alpha=1}^m R(JX, \eta_{\bar{\alpha}}) \eta_\alpha + R(JX, \eta_\alpha) \eta_{\bar{\alpha}} \right) \\
= & - J \left( \sum_{B=1}^{2m} R (JX, e_B) e_B \right) \numberthis \label{c-1}
\end{align*}
Substituting \eqref{c-1} into \eqref{c-3}, we obtain 
\begin{align*}
\langle R_* X, Y \rangle =\sum_{B=1}^{2m} \langle R (e_B, J X) JY, e_B \rangle + 2 (m-1) A(J X, Y).
\end{align*}
By replacing $X, Y$ by $JX, JY$, the proof is finished.
\end{proof}

For any $Y \in HM$, using \eqref{c-2}, we have
\begin{align*}
\sum_{B=1}^{2m} \langle \hat{R} (e_B, \xi) Y, e_B \rangle = \sum_{B=1}^{2m} \langle R (e_B, \xi) Y, e_B \rangle
\end{align*}
and 
\begin{align*}
\sum_{B=1}^{2m} \langle \hat{R}(e_B, Y) \xi, e_B \rangle = \sum_{B=1}^{2m} \langle S(e_B, Y) ,e_B \rangle = \mbox{div } \tau (Y).
\end{align*}
Applying the symmetric property of Riemannian curvature, we get
\begin{align}
\sum_{B=1}^{2m} \langle R (e_B, \xi) Y, e_B \rangle = \mbox{div } \tau (Y). \label{c-5}
\end{align}
Combing Lemma \ref{c-riemric1}, Lemma \ref{c-ricrelation} and \eqref{c-5}, we obtain the following lemma.

\begin{lemma} \label{c-lem-ric-1}
For any $X, Y \in TM$, we have
\begin{align*}
\sum_{B=1}^{2m} \langle \hat{R} (e_B, X) Y, e_B \rangle 
& = \langle R_* X, Y \rangle + 2 (m-1) A(X, JY) + \langle \tau X, \tau Y \rangle - 3 \langle \pi_H X, \pi_H Y \rangle   \\
& \quad  + (2m - |\tau|^2) \theta (X) \theta (Y) + \mbox{\normalfont div } \tau (X) \theta (Y) + \mbox{\normalfont div } \tau (Y) \theta (X) \numberthis  \label{c-riemric2}
\end{align*}
\end{lemma}

Hence Theorem \ref{c-thm-sub} can be obtained by Lemma \ref{c-lem1} and Lemma \ref{c-lem-ric-1}.

% Using Lemma \ref{c-lem1}, we obtain the following sub-Laplacian estimate of Riemannian distance function.

% \constantnumber{cst-2}

% \begin{theorem} \label{c-thm-sub}
% Suppose $(M^{2m +1}, \theta)$ is a complete pseudo-Hermitian manifold and $B_R (x_0)$ is the geodesic ball of radius $R$ centered at $x_0$. If 
% \begin{align*}
% R_* \geq - k, \mbox{ and } |A|, | \mbox{div} A | \leq k_1 , \quad \mbox{ on } B_R (x_0),
% \end{align*}
% then there exists $ C_{\ref*{cst-2}} = C_{\ref*{cst-2}} (m)$ such that
% \begin{align}
% \Delta_b r \leq C_{\ref*{cst-2}} \left(\frac{1}{r} + \sqrt{1 + k + k_1 + k_1^2}  \right), \quad \mbox{ on } B_R (x_0) \setminus Cut(x_0)  \label{c-riemdiscprs}
% \end{align}
% where $Cut (x_0)$ is the cut locus of $x_0$.
% \end{theorem}

\section{Horizontal Gradient Estimates} \label{sec-gradient}

Suppose that $(M^{2m+1}, \theta)$ is a complete noncompact pseudo-Hermitian manifold.
Let $r$ be the Riemannian distance function from $x_0 \in M$ associated with the Webster metric $g_\theta$ and $B_R$ be the geodesic ball of radius $R$ centered at $x_0$.
Assume that
\begin{align*}
R_* \geq - k, \mbox{ and } |A|, | \mbox{div} A | \leq k_1, \quad \mbox{ on } B_{2R}
\end{align*}
for some $R \geq 1$.
% Theorem \ref{c-thm-sub} shows that 
% \begin{align}
% \Delta_b r \leq C_{\ref*{cst-subcomparison}} ( \frac{1}{r} + 1 ), \quad \mbox{ on } B_{2 R} \setminus Cut (x_0),
% \end{align}
% \constantnumber{cst-subcomparison}
% where $C_{\ref*{cst-subcomparison}} = C_{\ref*{cst-subcomparison}} (m, k, k_1)$.
\constantnumber{cst-cutoff}
Choose a cut-off function $\varphi \in C^\infty ([0,\infty)) $ such that
\begin{align*}
\varphi \big|_{[0,1]} =1 , \quad  \varphi \big|_{[2, \infty)} =0,  \quad - C_{\ref*{cst-cutoff}}' | \varphi |^{\frac{1}{2}} \leq \varphi' \leq 0,
\end{align*}
where $C_{\ref*{cst-cutoff}}'$ is a universal constant.
By defining $ \chi (r) = \varphi (\frac{r}{R})$ and using Theorem \ref{c-thm-sub}, we find that
\begin{align}
\frac{|\nabla_b \chi|^2}{\chi} \leq \frac{C_{\ref*{cst-cutoff}}}{R^2} , \quad \Delta_b \chi \geq - \frac{C_{\ref*{cst-cutoff}}}{R}, \qquad \mbox{ on } B_{2 R} \setminus Cut (x_0), \label{d-cutoff}
\end{align}
where $C_{\ref*{cst-cutoff}} = C_{\ref*{cst-cutoff}} (m, k, k_1) $. 

Suppose that $(N, h)$ is a Riemannian manifold with sectional curvature 
\begin{align*}
K^N \leq \kappa
\end{align*}
for some $\kappa \geq 0$.
Denote the Riemannian distance function from $p_0 \in N$ by $\rho$.
Let $ B_D = B_D (p_0)$ be a regular ball of radius $D$ around $p_0$, that is $D < \frac{\pi}{2 \sqrt{\kappa}}$ and $B_D$ lies inside the cut locus of $p_0$ where $\frac{\pi}{2 \sqrt{\kappa}} = + \infty$ if $\kappa = 0$.
% Let $D < \frac{\pi}{2 \sqrt{\kappa}}$ for $\kappa > 0$ or $D < + \infty$ for $\kappa =0$ such that 
% \begin{align*}
% 	B_D = B_D (p_0) = \{ q \in N \big| \: \rho (q) < D \}
% \end{align*}
% lie inside the cut locus of $p_0$. 
Set
\begin{align*}
\phi (t) = 
\begin{cases}
\frac{1- \cos (\sqrt{\kappa} t )}{\kappa}, & \kappa >0 \\
\frac{t^2}{2}, & \kappa =0
\end{cases}
.
\end{align*}
and 
\begin{align*}
\psi (q) = \phi \circ \rho (q).
\end{align*}
Obviously, $\phi$ is an increasing function and $\psi$ is at least $C^2$ in the cut locus of $p_0$.
Moreover, Hessian comparison theorem shows that
\begin{align}
\mbox{Hess } \psi \geq \cos ( \sqrt{\kappa} \rho) \cdot h. \label{d-comparison}
\end{align}

\begin{lemma} \label{d-lemma-nu-b}
For any $0 < D < \frac{\pi}{2 \sqrt{\kappa}}$, there exist $\nu \in [1,2)$, $b > \phi (D)$ and $\delta >0$ only depending on $D$ such that 
\begin{align}
\nu \frac{\cos (\sqrt{\kappa} t)}{b - \phi(t)} -2 \kappa > \delta, \quad  \forall t \in [0, D]
\end{align}
\end{lemma}

\begin{proof}
For the case $\kappa > 0$, it suffices to find $\nu \in [1,2)$ and $b > \phi (D)$ such that
\begin{align}
\phi (D) < b < \inf_{s \in [0, \phi (D)]} \left( \frac{\nu}{2 \kappa} + ( 1- \frac{\nu }{2}) s \right), \label{d-10}
\end{align}
which is obvious due to $\phi (D) < \frac{1}{\kappa}$.

The case $\kappa = 0$ is obvious by choosing $\nu =1$.
\end{proof}

Assume that $f: B_{2 R} (x_0) \subset M \to B_D (p_0) $ is a pseudo-harmonic map. By \eqref{d-comparison}, we have the following estimate: 

\begin{lemma} \label{d-lemma-comparison}
Let $\nu , b, \delta$ be given in Lemma \ref{d-lemma-nu-b}. Then
\begin{align}
\nu \frac{\Delta_b \psi \circ f}{b - \psi \circ f} - 2 \kappa |d_b f|^2 \geq \delta |d_b f|^2
\end{align}
\end{lemma}

To estimate $|d_b f|^2$, we consider the following auxiliary function
\begin{gather*}
\Phi_{\mu \chi} = |d_b f|^2 + \mu \chi |f_0|^2
\end{gather*}
where $\mu$ will be determined later.

\begin{lemma} \label{d-bochner-phi}
Suppose $\mu$ and $\epsilon$ satisfy
\begin{align*}
C_{\ref*{cst-reebcrbochner}} \mu \leq \epsilon \leq 1.
\end{align*}
If $\chi(x) \neq 0$ and $\Phi_{\mu \chi} (x) \neq 0$,
then at $x$, we have
\begin{align} \label{d-bochnerestimate}
\Delta_b \Phi_{\mu \chi} \geq & \frac{1- \epsilon}{2} \frac{|\nabla_b \Phi_{\mu \chi}|^2}{\Phi_{\mu \chi}}  - 2 \kappa |d_b f|^2 \Phi_{\mu \chi} \nonumber \\
& + \left(2m \epsilon - C_{\ref*{cst-reebcrbochner}} \mu \chi - 4 \epsilon^{-1} \mu \chi^{-1} |\nabla_b \chi|^2 + \mu \Delta_b \chi \right) |f_0|^2 \nonumber \\
& -  \left[C_{\ref*{cst-reebcrbochner}} + C_{\ref*{cst-reebcrbochner}} \mu \chi + 16 (\epsilon \mu \chi)^{-1} \right] |d_b f|^2 
\end{align}
\end{lemma}

\begin{proof}
% The estimate \eqref{b-bochner2} gives that
% \begin{align*}
% \Delta_b (\chi |f_0|^2) =& \chi \Delta_b |f_0|^2 + 2 \langle \nabla_b \chi, \nabla_b |f_0|^2 \rangle + |f_0|^2 \Delta_b \chi \\
% \geq & 2 \chi |\nabla_b f_0|^2 -2 \kappa \chi |f_0|^2 |d_b f|^2 - C_{\ref*{cst-reebcrbochner}} \chi | \pi_{(1,1)}^{\perp} \nabla_b d_b f |^2 - C_{\ref*{cst-reebcrbochner}} \chi |f_0|^2 - C_{\ref*{cst-reebcrbochner}} \chi |d_b f|^2 \\ 
% & + 4 \langle \nabla_b \chi \otimes f_0, \nabla_b f_0 \rangle + |f_0|^2 \Delta_b \chi. \numberthis \label{d-reeb}
% \end{align*}
Using \eqref{b-bochner1} and \eqref{b-bochner2} with $\epsilon_1 = \epsilon \mu \chi$, we have 
\begin{align*}
\Delta_b \Phi_{\mu \chi} =& \Delta_b (|d_b f|^2 + \mu \chi |f_0|^2) \\
\geq & (2- \epsilon) (|\nabla_b d_b f|^2 +  \mu \chi |\nabla_b f_0|^2) + 4 \mu \langle \nabla_b \chi \otimes f_0, \nabla_b f_0 \rangle -2 \kappa \Phi_{\mu \chi} |d_b f|^2   \\
&+ \left[ 2m \epsilon -C_{\ref*{cst-reebcrbochner}} \mu \chi + \mu \Delta_b \chi \right] |f_0|^2 - \left[C_{\ref*{cst-reebcrbochner}} + C_{\ref*{cst-reebcrbochner}} \mu \chi + 16 (\epsilon \mu \chi)^{-1} \right] |d_b f|^2 \numberthis \label{d-7}
\end{align*}
By Cauchy inequality, we have the following estimate
\begin{align*}
|\nabla_b \Phi_{\mu \chi}|^2 & = |\nabla_b (|d_b f|^2 + \mu \chi |f_0|^2) |^2 \\
& = | \nabla_b \langle d_b f + \sqrt{\mu \chi} f_0 \otimes \theta, d_b f + \sqrt{\mu \chi} f_0 \otimes \theta  \rangle |^2 \\
& = 4 \left| \left\langle d_b f + \sqrt{\mu \chi} f_0 \otimes \theta, \nabla_b d_b f + \sqrt{\mu \chi} \nabla_b f_0 \otimes \theta + \sqrt{\mu } \frac{\nabla_b \chi}{2 \sqrt{\chi}} \otimes f_0 \otimes \theta \right\rangle \right|^2 \\
& \leq 4 \big| d_b f + \sqrt{\mu \chi} f_0 \otimes \theta \big|^2  \cdot \left| \nabla_b d_b f + \sqrt{\mu \chi} \nabla_b f_0 \otimes \theta + \sqrt{\mu} \frac{\nabla_b \chi}{2 \sqrt{\chi}} \otimes f_0 \otimes \theta \right|^2 \\
&= 4 \Phi_{\mu \chi} \left(|\nabla_b d_b f|^2 + \mu \chi |\nabla_b f_0|^2 + \frac{\mu |\nabla_b \chi|^2 }{4 \chi} |f_0|^2 + \mu \langle \nabla_b f_0, \nabla_b \chi \otimes f_0 \rangle  \right) 
\end{align*}
which, using Cauchy inequality again, implies that 
\begin{align*}
& (2- \epsilon) (|\nabla_b d_b f|^2 +  \mu \chi |\nabla_b f_0|^2) + 4 \mu \langle \nabla_b \chi \otimes f_0, \nabla_b f_0 \rangle \\
& = (2- 2 \epsilon) \left( |\nabla_b d_b f|^2 +  \mu \chi |\nabla_b f_0|^2 \right) + \epsilon \mu \chi |\nabla_b f_0|^2 + 4 \mu \langle \nabla_b \chi \otimes f_0, \nabla_b f_0 \rangle \\
& \geq \frac{1-\epsilon}{2} \frac{|\nabla_b \Phi_{\mu \chi}|^2}{\Phi_{\mu \chi}} - \frac{1-\epsilon}{2} \frac{\mu |\nabla_b \chi|^2 }{\chi} |f_0|^2 + (2 + 2 \epsilon) \mu \langle \nabla_b \chi \otimes f_0, \nabla_b f_0 \rangle + \epsilon \mu \chi |\nabla_b f_0|^2  \\
& \geq \frac{1-\epsilon}{2} \frac{|\nabla_b \Phi_{\mu \chi}|^2}{\Phi_{\mu \chi}} - \left( \frac{1-\epsilon}{2} + \frac{(1+ \epsilon)^2}{\epsilon} \right) \mu \frac{ |\nabla_b \chi|^2 }{\chi} |f_0|^2  \\
& \geq \frac{1-\epsilon}{2} \frac{|\nabla_b \Phi_{\mu \chi}|^2}{\Phi_{\mu \chi}} - 4 \epsilon^{-1} \mu \frac{ |\nabla_b \chi|^2 }{\chi} |f_0|^2   \numberthis \label{d-6}
\end{align*}
due to $ \epsilon \leq 1$ and 
\begin{align*}
\frac{1-\epsilon}{2} + \frac{(1+ \epsilon)^2}{\epsilon} \leq \frac{1-\epsilon}{\epsilon} + \frac{(1+ \epsilon)^2}{\epsilon} = 2 \epsilon^{-1} + \epsilon + 1 \leq 4 \epsilon^{-1}.
\end{align*}
Submitting \eqref{d-6} to \eqref{d-7}, we finished the proof.
\end{proof}

\begin{proof}[Proof of Theorem \ref{d-thm-estimate}]
Set
\begin{align*}
F_{\mu \chi} = \frac{\Phi_{\mu \chi}}{(b - \psi \circ f)^\nu}
\end{align*}
where $\nu \in [1,2)$ and $b$ are determined in Lemma \ref{d-lemma-nu-b}.
The $\epsilon$ in Lemma \ref{d-bochner-phi} is chosen as
\begin{align}
\epsilon = \frac{1}{\nu} - \frac{1}{2} \leq 1 \label{d-epsilon}
\end{align}
and $\mu$ satisfy
\begin{align}
C_{\ref*{cst-reebcrbochner}} \mu \leq \epsilon .  \label{d-epsilon-mu}
\end{align}
Let $x$ be a maximum point of $\chi F_{\mu \chi}$ on $B_{2 R}$ which is nonzero. Assume that $r$ is smooth at $x$. Otherwise we can modify the distance function $r$ as \cite{cheng1980liouville}. Hence at $x$, we have
\begin{align}
0 = \nabla_b \ln (\chi F_{\mu \chi}) &= \frac{\nabla_b \chi}{\chi} + \frac{\nabla_b \Phi_{\mu \chi}}{\Phi_{\mu \chi}} + \nu \frac{\nabla_b (\psi \circ f)}{b - \psi \circ f} , \label{d-1} \\
0 \geq \Delta_b \ln (\chi F_{\mu \chi}) & = \frac{\Delta_b \chi}{\chi} - \frac{|\nabla_b \chi|^2}{\chi^2} + \frac{\Delta_b \Phi_{\mu \chi}}{\Phi_{\mu \chi}} - \frac{|\nabla_b \Phi_{\mu \chi}|^2}{\Phi_{\mu \chi}^2} \nonumber \\
& \quad + \nu \frac{\Delta_b (\psi \circ f)}{b - \psi \circ f} + \nu \frac{|\nabla_b (\psi \circ f)|^2}{(b - \psi \circ f)^2}. \label{d-2}
\end{align}
By \eqref{d-bochnerestimate}, \eqref{d-2} becomes
\begin{align*}
0 \geq & \frac{\Delta_b \chi}{\chi} - \frac{|\nabla_b \chi|^2}{\chi^2} - \frac{1 + \epsilon}{2} \frac{|\nabla_b \Phi_{\mu \chi}|^2}{\Phi_{\mu \chi}^2} - 2 \kappa |d_b f|^2 + \nu \frac{\Delta_b (\psi \circ f)}{b - \psi \circ f} + \nu \frac{|\nabla_b (\psi \circ f)|^2}{(b - \psi \circ f)^2}  \\
&+ \left( 2m \epsilon -C_{\ref*{cst-reebcrbochner}} \mu \chi + \mu \Delta_b \chi - 4 \epsilon^{-1} \mu \frac{ |\nabla_b \chi|^2 }{\chi} \right) \frac{|f_0|^2 }{\Phi_{\mu \chi}} - \left[C_{\ref*{cst-reebcrbochner}} + C_{\ref*{cst-reebcrbochner}} \mu \chi + 16 (\epsilon \mu \chi)^{-1} \right] \frac{|d_b f|^2}{\Phi_{\mu \chi}}. \numberthis \label{d-5}
\end{align*}
Using \eqref{d-1} and Cauchy inequality, we have at $x$
\begin{align}
- \frac{1+ \epsilon}{2} \frac{|\nabla_b \Phi_{\mu \chi}|^2}{\Phi_{\mu \chi}^2} 
\geq - \frac{1+ \epsilon}{2} (1 + \epsilon_2^{-1}) \frac{| \nabla_b \chi |^2}{\chi^2} -  \frac{1+ \epsilon}{2} (1+ \epsilon_2) \nu^2 \frac{|\nabla_b (\psi \circ f)|^2}{(b - \psi \circ f)^2}. \label{d-3}
\end{align}
Due to the choice \eqref{d-epsilon} of $\epsilon$, we can take
\begin{align*}
\epsilon_2 = \frac{2}{\nu (1+ \epsilon)} -1 = \frac{2-\nu}{2+\nu} >0
\end{align*}
and then
\begin{gather}
\frac{1+ \epsilon}{2} (1+ \epsilon_2) \nu^2 = \nu , \qquad
\frac{1+ \epsilon}{2} (1 + \epsilon_2^{-1}) = \frac{2+ \nu}{\nu(2-\nu)}.
\end{gather}
Substituting \eqref{d-3}, \eqref{d-comparison} to \eqref{d-5}, we have at $x$
\begin{align*}
0 \geq & \frac{\Delta_b \chi}{\chi} - \left( 1+ \frac{2+ \nu}{\nu(2-\nu)} \right) \frac{|\nabla_b \chi|^2}{\chi^2} + \nu \frac{\Delta_b \psi \circ f}{b - \psi \circ f} - 2 \kappa |d_b f|^2 \\
&+ \left( 2m \epsilon -C_{\ref*{cst-reebcrbochner}} \mu \chi + \mu \Delta_b \chi - 4 \epsilon^{-1} \mu \frac{ |\nabla_b \chi|^2 }{\chi} \right) \frac{|f_0|^2 }{\Phi_{\mu \chi}} - \left[C_{\ref*{cst-reebcrbochner}} + C_{\ref*{cst-reebcrbochner}} \mu \chi + 16 (\epsilon \mu \chi)^{-1} \right] \frac{|d_b f|^2}{\Phi_{\mu \chi}}.
\end{align*}
The estimates \eqref{d-cutoff} and Lemma \ref{d-lemma-comparison} yield that
\begin{align}
0 \geq - \frac{C_{\nu}}{\chi R}  + \delta |d_b f|^2 + \left( 2m \epsilon -C_{\ref*{cst-reebcrbochner}} \mu \chi - \frac{\mu C_{\nu}}{R} \right) \frac{|f_0|^2 }{\Phi_{\mu \chi}} - \left[C_{\ref*{cst-reebcrbochner}} + C_{\ref*{cst-reebcrbochner}} \mu \chi + 16 (\epsilon \mu \chi)^{-1} \right] \frac{|d_b f|^2}{\Phi_{\mu \chi}}, \label{d-11}
\end{align}
where $C_\nu = C_\nu ( \nu, C_{\ref*{cst-cutoff}} ) $ and $\delta$ is given by Lemma \ref{d-lemma-nu-b}.
By definition of $\Phi_{\mu \chi}$,
\begin{align*}
|f_0|^2  = \mu^{-1} \chi^{-1} (\Phi_{\mu \chi} - |d_b f|^2)
\end{align*}
which, together with \eqref{d-11}, shows at $x$,
\begin{align}
0 \geq \frac{1}{\chi} \left( 2m \epsilon \mu^{-1} -C_{\ref*{cst-reebcrbochner}} - \frac{ 2 C_{\nu}}{R} \right)  + \bigg[ \delta \chi \Phi_{\mu \chi} - 2m \epsilon \mu^{-1} - \left[C_{\ref*{cst-reebcrbochner}} + C_{\ref*{cst-reebcrbochner}} \mu + 16 (\epsilon \mu)^{-1} \right] \bigg] \frac{|d_b f|^2}{\chi \Phi_{\mu \chi}} \label{d-4}
\end{align}
To make the first bracket of the last line in \eqref{d-4} nonnegative, we can choose sufficiently small $\mu$ such that 
\begin{align*}
\epsilon \mu^{-1} = C_{\ref*{cst-reebcrbochner}}  + \frac{ 2 C_{\nu}}{R},
\end{align*}
which makes \eqref{d-epsilon-mu} right.
Hence 
\constantnumber{cst-bound}
\begin{align}
(\chi \Phi_{\mu \chi}) (x) \leq C_{\ref*{cst-bound}} \delta^{-1}  , \label{d-8}
\end{align}
where
\begin{align}
C_{\ref*{cst-bound}} = (2m+1) C_{\ref*{cst-reebcrbochner}} + \frac{4m C_\nu}{R} + \frac{C_{\ref*{cst-reebcrbochner}}}{2 C_{\ref*{cst-reebcrbochner}} + 4 C_\nu R^{-1}} + \frac{64 \nu^2}{(2-\nu)^2} \left( C_{\ref*{cst-reebcrbochner}} + \frac{2 C_\nu}{R} \right), \label{d-9}
\end{align}
which implies
\begin{align}
\max_{B_{2R} (x_0)} \chi F_{\mu \chi} \leq \frac{\chi \Phi_{\mu \chi}}{(b- \psi \circ f)^\nu} (x) \leq \frac{C_{\ref*{cst-bound}}}{\delta (b- \phi(D))^\nu }.
\end{align}
This shows that 
\begin{align}
\max_{B_R (x_0)} |d_b f|^2 \leq b^\nu \cdot \max_{B_R (x_0)} F_{\mu \chi} \leq  \frac{C_{\ref*{cst-bound}} b^{\nu}}{\delta (b- \phi (D))^{\nu}}.
\end{align}
Note that the constants $b, \nu $ and $ \delta$ depend on $\kappa$ and $D$ by Lemma \ref{d-lemma-nu-b}.
Hence the proof is finished by choosing a suitable constant $C_{\ref*{cst-3}}$.
\end{proof}

% Let's summarize the results as follows.

% \begin{proof}[Proof of Theorem \ref{d-thm-liouville}]
% ....
% \end{proof}

% Lemma \ref{b-lem-estimate} says that if $k =0$ and $k_1 =0$, then $C_{\ref*{cst-reebcrbochner}} =0$. 
% According to \eqref{d-9}, we find that
% \begin{align*}
% C_{\ref*{cst-bound}} = \left(4m + \frac{128 \nu^2}{m (2- \nu)^2} \right) \frac{C_{\nu}}{R},
% \end{align*}
% which implies that 
% \begin{align*}
% \max_{B_R (x_0)} |d_b f|^2 \leq \left( 4m + \frac{128 \nu^2}{m (2- \nu)^2} \right) \frac{C_\nu}{R} \frac{b^{\nu}}{\delta (b- \phi (D))^{\nu}} \to 0 , \mbox{ as } R \to \infty.
% \end{align*}
% Hence we have the following Liouville theorem of pseudo-harmonic maps.

% \begin{theorem} 
% Let $(M, \theta)$ be a noncompact complete Sasakian manifold with nonnegative pseudo-Hermitian Ricci curvature and $(N, h)$ be a Riemannian manifold with sectional curvature bounded above. Then there is no nontrivial pseudo-Hermitian map from $M$ to any regular ball of $N$. 
% \end{theorem}

\section{Global Existence Theorem} \label{sec-existence}

Jost and Xu \cite{jost1998subelliptic} studied the minimizing sequence of Dirichlet problem of subelliptic harmonic maps and obtained the existence theorem under some convexity condition. 
Their results \cite{jost1998subelliptic} seem to depend on the global fields which satisfy the H\"ormander condition and the noncharacteristic assumption of the boundary. 
But the weak existence of Dirichlet problem and the interior continuity of weak solutions can be generalized to any sub-Riemannian manifolds with smooth boundaries, such as pseudo-Hermitian manifolds.
Hence Theorem 1 in \cite{jost1998subelliptic} can be generalized to pseudo-Hermitian manifolds with boundary as follows.

\begin{theorem} \label{e-thm-dirichlet}
Suppose that $(M,\theta)$ is a pseudo-Hermitian manifold with smooth boundary and $(N,h)$ is a Riemannian manifold with sectional curvature $K^N \leq \kappa$ for some $\kappa \geq 0$. Let $B_D = B_D (p_0) \subset N$ be a regular ball. If $\varphi \in S^2_1 (M, N)$ satisfies $\varphi (\overline{M}) \subset B_D (p_0)$, then there exists a weak pseudo-harmonic map $f \in C (M, N) \cap S^2_1 (M, N) $ with
\begin{align*}
f- \varphi \in S^2_{1, 0} (M, N)
\end{align*}
and 
\begin{align*}
f (\overline{M}) \subset B_D (p_0).
\end{align*}
\end{theorem}

For completeness, the proof will been given in Appendix.

\begin{remark}
Note that $B_D (p_0)$ can be covered by a geodesic normal coordinate $\{ z^i \}$ and thus it can be viewed as an open set of $\mathbb{R}^n$ where $n = \mbox{dim } N$. Hence the notion
\begin{align*}
S^2_1 (M, N) = S^2_1 (M, \mathbb{R}^n),
\end{align*}
and $S^2_{1,0} (M, N)$ means the completion of all smooth $\mathbb{R}^n$-valued functions with compact support under $S^2_1$-norm. Moreover, the weak pseudo-harmonic map $f \in S^2_1 (M, N)$ means that the following equations hold in the distribution sense 
\begin{align}
\Delta_b f^i + \sum_{j, k} \Gamma^i_{jk} (f) \langle \nabla_b f^j, \nabla_b f^k \rangle = 0, \quad \mbox{for all $i = 1, 2, \dots n$,} \label{e-pse-har}
\end{align}
where $f^i = z^i \circ f$ and $\Gamma^i_{jk}$'s are Christoffel symbols of Levi-Civita connection in $(N, h)$.
\end{remark}

Since the Euler-Lagrange equations of pseudo-harmonic maps are quasilinear subelliptic systems, these weak solutions will be interior smooth by applying Theorem 1.1 in \cite{xu1997higher} to each coordinate neighborhood.

\begin{theorem} \label{e-thm-smooth}
Suppose that $(M, \theta)$ is a pseudo-Hermitian manifold (with or without boundary) and $(N,h)$ is a Riemannian manifold. Let $f : M \to N$ be a weak pseudo-harmonic map and $f \in S^2_1 (M , N)$. If $f$ is continuous inside $M$, then $f \in C^\infty (M, N)$.
\end{theorem}

% Now we consider the global existence of pseudo-harmonic maps to regular balls. 
Now let's come to prove Theorem \ref{e-thm-exist}.
\begin{proof}[Proof of Theorem \ref{e-thm-exist}]
Suppose that $(M, \theta)$ is a complete noncompact pseudo-Hermitian manifold and $(N, h)$ is a Riemannian manifold with sectional curvature $K^N \leq \kappa$ for some $\kappa \geq 0$. Let $B_D (p_0) \subset N$ be a geodesic ball lying in the cut locus of $p_0$ and $D < \frac{\pi}{2 \sqrt{\kappa}}$. Assume that $\varphi : M \to B_D(p_0) $ with $\varphi (x_0) = p_0$.
We can choose a smooth exhaustion $\{ \Omega_i \}$ of $M$ such that $B_{2i} (x_0) \subset \Omega_i$.
Theorem \ref{e-thm-dirichlet} and Theorem \ref{e-thm-smooth} guarantee that there is a smooth pseudo-harmonic map $f_i : \Omega_i \to B_D (p_0)$.
One can find the constants $k(i)$ and $k_1 (i)$ such that
\begin{align}
R_* \big|_{B_{2i} (x_0)} \geq - k(i), \mbox{ and } \big|A |_{B_{2i} (x_0)} \big|, \big| \mbox{div } A |_{B_{2i} (x_0)} \big|  \leq k_1 (i).
\end{align}
Hence fixed $i$, for $j \geq i$, Theorem \ref{d-thm-estimate} controls the interior horizontal gradient of $f_j$ on $B_i (x_0)$: 
\constantnumber{cst-estimate}
\begin{align}
\max_{B_i (x_0)} |d_b f_j|^2  \leq C_{\ref*{cst-estimate}} ( i),
\end{align}
where $C_{\ref*{cst-estimate}} ( i)$ only depends on $k (i), k_1 (i), D, \kappa, i$. 
Arzel\`a-Ascoli theorem yields that by taking subsequence, $f_j$ will uniformly converge to some continuous map in $B_i (x_0)$ as $j \to \infty$.
By diagonalization, some subsequence of $\{ f_i \}$ will internally closed uniformly converge to a continuous map $f : M \to B_D (p_0)$ as $i \to \infty$.
Moreover, $f$ is a weak solution of \eqref{e-pse-har} and thus is smooth pseudo-harmonic by Theorem \ref{e-thm-smooth}.
\end{proof}

% \begin{theorem} \label{e-thm-exist}
% Suppose that $(M, \theta)$ is a complete noncompact pseudo-Hermitian manifold and $(N, h)$ is a Riemannian manifold with sectional curvature bounded from above. Let $B_D (p_0) \subset N$ be a regular ball. Then there is a pseudo-harmonic map $f : M \to B_D (p_0)$.
% \end{theorem}

It is notable that the pseudo-harmonic map given by Theorem \ref{e-thm-exist} will depend on the initial map.  
By Theorem \ref{d-thm-liouville}, it is always trivial if the domain has nonnegative pseudo-Hermitian Ricci curvature. 
At the end of this paper, we will give a nontrivial example when the domain has negative pseudo-Hermitian Ricci curvature.
One model of Sasakian space form with constant negative pseudo-Hermitian sectional curvature is the Riemannian submersion
\begin{align*}
\pi : B^n_{\mathbb{C}} \times \mathbb{R} \to B^n_{\mathbb{C}}
\end{align*}
where $B^n_{\mathbb{C}} \subset \mathbb{C}^n $ is the complex ball with Bergman metric $\omega$ (cf. Example 7.3.22 in \cite{boyer2008sasakian}). 
Let $\omega_0$ be the canonical K\"ahler form on $\mathbb{C}^n$. 
Since the identity $I$ of $B^n_{\mathbb{C}}$ is a holomorphic map from $B^n_{\mathbb{C}}$ to $\mathbb{C}^n$, then it is also a harmonic map from $(B^n_{\mathbb{C}}, \omega)$ to $(\mathbb{C}^n, \omega_0)$.
The lift of $I$ is denoted by $\tilde{I}$ such that
\begin{align*}
\tilde{I} = I \circ \pi : B^n_{\mathbb{C}} \times \mathbb{R} \to \mathbb{C}^n.
\end{align*}
Then by the composition rule,
\begin{align*}
\hat{\nabla} d \tilde{I} = \hat{\nabla} d I (d \pi, d \pi) + d I (\hat{\nabla} d \pi) \numberthis \label{e-1}
\end{align*}
where the Levi-Civita connections of $(B^n_{\mathbb{C}}, \omega)$ and $(\mathbb{C}^n, \omega_0)$ are both denoted by $\hat{\nabla}$.
Suppose that $\nabla$ is the Tanaka-Webster connection of $B^n_{\mathbb{C}} \times \mathbb{R}$.
Their relation is given by (cf. Lemma 1.3 in \cite{dragomir2006differential})
\begin{align}
\hat{\nabla} = \nabla - d \theta \otimes \xi + 2 \theta \odot J \label{e-connection}
\end{align}
where $2 \theta \odot J = \theta \otimes J + J \otimes \theta$.
Assume that $\{ e_B \}_{B=1}^{2n}$ is a orthonormal frame in $(B^n_{\mathbb{C}}, \omega)$ with $e_{\alpha + n} = J e_\alpha$ for $1 \leq \alpha \leq n$ and $\tilde{e}_B$ is the horizontal lift of $e_B$. 
On one hand, the relation \eqref{e-connection} guarantees that
\begin{align*}
\tau_H (\tilde{I}) &= \sum_{B=1}^{2n} (\nabla_{\tilde{e}_B} d \tilde{I}) (\tilde{e}_B) \\
& = \sum_{B=1}^{2n} \hat{\nabla}_{\tilde{e}_B} \left( d \tilde{I} (\tilde{e}_B) \right) - \sum_{B=1}^{2n} d \tilde{I} \left( \nabla_{\tilde{e}_B} \tilde{e_B} \right) \\
& = \sum_{B=1}^{2n} \hat{\nabla}_{\tilde{e}_B} \left( d \tilde{I} (\tilde{e}_B) \right) - \sum_{B=1}^{2n} d \tilde{I} \left( \hat{\nabla}_{\tilde{e}_B} \tilde{e_B} \right) \\
& = \sum_{B=1}^{2n} (\hat{\nabla}_{\tilde{e}_B} d \tilde{I}) (\tilde{e}_B). \numberthis \label{e-2}
\end{align*}
On the other hand, by the relation of Levi-Civita connection and metric, we have
\begin{align*}
\sum_{B=1}^{2n} d \pi \left( \hat{\nabla}_{\tilde{e}_B} \tilde{e}_B \right) = \sum_{B=1}^{2n} \hat{\nabla}_{e_B} e_B
\end{align*}
which implies that
\begin{align*}
\sum_{i=1}^{2n} \left( \hat{\nabla}_{\tilde{e}_B} d \pi \right) (\tilde{e}_B)   = 0 . \numberthis \label{e-3}
\end{align*}
Taking the horizontal trace of \eqref{e-1} and using \eqref{e-2}, \eqref{e-3}, we obtain that
\begin{align*}
\tau_H (\tilde{I}) = \sum_{i=1}^{2n} \left( \hat{\nabla}_{e_B} d I \right) (e_B)  =0,
\end{align*}
since $I$ is harmonic.
Hence $\tilde{I}$ is nontrivial pseudo-harmonic.
But the image of $\tilde{I}$ is exactly the unit ball in $\mathbb{C}^n$ which is a regular ball.
So this is a nontrivial pseudo-harmonic example when the domain has negative pseudo-Hermitian Ricci curvature.

\section{Appendix}
This section will deduce Theorem \ref{b-lem-regular} and Theorem \ref{e-thm-dirichlet} by the theory of subelliptic analysis.
Suppose that $ (M, \theta) $ is a pseudo-Hermitian manifold of real dimension $2m +1$. 
Let $\Omega$ be a coordinate neighborhood in $M$ and $\{ e_B \}_{B=1}^{2m}$ be an orthonormal basis of $HM \big|_\Omega$ with $J e_i = e_{i +m}$ for $i = 1, 2, \dots, m$.
Since 
\begin{align*}
- \theta ([e_i, J e_i]) = d \theta (e_i, J e_i) = G_\theta (e_i, e_i) = 1, \quad \mbox{for } i = 1, 2, \dots, m,
\end{align*}
then each $[e_i, J e_i]$ is transversal with horizontal distribution which implies that $HM$ satisfies the strong bracket generating hypothesis.
Moreover, by identifying $\Omega$ with a domain in $\mathbb{R}^{2m+1}$, the vector fields $\{e_1, \dots, e_{2m} \}$ satisfy the H\"ormander's condition.
Let $e_B^*$ be the formal adjoint of $e_B$. 
For any $u \in C^\infty (\Omega)$, we have 
\begin{align*}
\Delta_b u = - \sum_{B =1}^{2m} e_B^* e_B u,
\end{align*}
which shows that the sub-Laplacian operator is subelliptic.
One can refer to Section 2.2 in \cite{dragomir2006differential} for more discussions.
Since Tanaka-Webster connection preserves the horizontal distribution, 
then the higher-order horizontal covariant derivative on $\Omega$ can be expressed as follows:
\begin{align*}
\nabla_b^l u (e_{B_1}, \cdots, e_{B_l}) 
& = \nabla_{e_{B_l}} \bigg[ \nabla^{l-1} u (e_{B_1}, \cdots, e_{B_{l-1}})  \bigg] - \sum_{i = 1}^{l} \nabla^{l-1} u (e_{B_1}, \cdots, \nabla_{e_{B_l}} e_{B_i}, \cdots, e_{B_l})  \\
& = \cdots \\
& = e_{B_l} e_{B_{l-1}} \cdots e_{B_1} u + \mbox{ lower order terms},
\end{align*}
for any $B_1, \cdots, B_l \in \{1, 2, \cdots , 2m\} $,
which implies that the $S^p_k$-norm on $\Omega$ is equivalent with the local Folland-Stein Sobolev norm (cf. Page 193 in \cite{dragomir2006differential}).
Hence local results of subelliptic analysis always hold for the sub-Laplacian operator on a coordinate neighborhood of pseudo-Hermitian manifolds.
By partition of unity, the domain can be generalized to a relatively compact domain in a pseudo-Hermitian manifold.
Let's use this idea to prove Theorem \ref{b-lem-regular} by the following local version.

\begin{theorem}[Theorem 3.17 in \cite{dragomir2006differential}, Theorem 16 in \cite{rothschild1976hypoelliptic}] \label{f-thm-local-1}
Suppose that $(M,\theta)$ is a pseudo-Hermitian manifold and $\Omega \Subset M$ is a coordinate neighborhood. Assume that $u, v \in L_{loc}^1 (\Omega)$ and $\Delta_b u = v$ in the distribution sense. For any $\chi \in C^\infty_0 (\Omega)$, if $v \in S^p_k (\Omega) $ with $p >1$ and $k \in \mathbb{N}$, then $\chi u \in S^p_{k+2} (\Omega)$ and 
\begin{align}
||\chi u||_{S^p_{k+2} (\Omega)} \leq C_{\chi} \left( ||u||_{L^p (\Omega)} + ||v||_{S^p_k (\Omega)} \right)
\end{align}
where $C_{\chi}$ only depends on $\chi$.
\end{theorem}

\begin{proof}[Proof of Theorem \ref{b-lem-regular}]
Let $\{\Omega_\alpha\}$ be a finite open cover of $ \operatorname{supp}  \chi$ and $\{\chi_\alpha\}$ be a partition of unity subordinating to $\{\Omega_\alpha\}$.
Since $\Delta_b u = v$ holds in each $\Omega_\alpha$, Theorem \ref{f-thm-local-1} guarantees that 
\begin{align*}
|| \chi_\alpha \chi u||_{S^p_{k+2} (\Omega_\alpha)} \leq C_{\chi_\alpha \chi} \left( ||u||_{L^p (\Omega_\alpha)} + ||v||_{S^p_k (\Omega_\alpha)} \right),
\end{align*}
which implies that
\begin{align*}
||\chi u||_{S^p_{k+2} (\Omega)} 
\leq \sum_\alpha ||\chi u||_{S^p_{k+2} (\Omega_\alpha)} 
\leq \left( \sum_\alpha C_{\chi_\alpha \chi} \right) \left( ||u||_{L^p (\Omega)} + ||v||_{S^p_k (\Omega)} \right).
\end{align*}
The proof is finished by setting $C_\chi = \sum_\alpha C_{\chi_\alpha \chi}$.
\end{proof}

Next let's prove Theorem \ref{e-thm-dirichlet}.

\begin{proof}[Proof of Theorem \ref{e-thm-dirichlet}]
Under the exponential map at $p_0 \in N$, the regular ball $B_D = B_D (p_0)$ is diffeomorphic to the ball $B_D$ with radius $D$ and centered at the origin in $\mathbb{R}^n$ where $n = \operatorname{dim} N$.
Let $\{z^i\}_{i=1}^n$ be the geodesic normal coordinates at $p_0$ and $f^i = z^i \circ f$ be the components of a function $f : M \to B_D$.
Denote
\begin{align*}
\mathcal{S} = \left\{ f \in S^2_1 (M, \mathbb{R}^n) \bigg| \: f - \varphi \in S^2_{1,0} (M, \mathbb{R}^n), \ \sup_M |f| \leq D \right\},
\end{align*}
where $|\cdot|$ is the Euclidean norm in $\mathbb{R}^n$.
Consider the minimizing problem
\begin{align}
  \lambda = \inf_{f \in \mathcal{S}} E_H (f) = \inf_{f \in \mathcal{S}} \int_M h_{ij} (f) \langle \nabla_b f^i, \nabla_b f^j \rangle \label{f-min}
\end{align}
where $ h_{ij} = h (\frac{\partial}{\partial z^i}, \frac{\partial}{\partial z^j}) $.
Since $ \varphi \in \mathcal{S} $, then $ \lambda $ is finite.
Let $ \{f_s\}_{s =1}^\infty $ be a minimizing sequence of \eqref{f-min} which have uniform $ S^2_1 $-norm bound.
By CR compact embedding theorem of Folland-Stein space (cf. Theorem 3.15 in \cite{dragomir2006differential}), there are a $ f \in S^2_1 (M , \mathbb{R}^n) $ and a subsequence of $ \{f_s\} $ (also denoted by $ \{f_s\} $) such that
\begin{enumerate}[(i)]
\item $ f_s \to f $ strongly in $ L^2 (M, \mathbb{R}^n) $; \label{f-f1}
\item $ f_s \rightharpoonup f $ weakly in $ S^2_1 (M, \mathbb{R}^n) $. \label{f-f2}
\end{enumerate}
By \eqref{f-f1}, $ f_s $ converges to $ f $ almost everywhere on $ M $ which implies that $ |f| \leq D $; 
by \eqref{f-f2}, $ f - \varphi \in S^2_{1,0} (M, \mathbb{R}^n) $ which is closed in $ S^2_1 (M, \mathbb{R}^n) $. 
Hence $ f \in \mathcal{S} $. 

We claim that
\begin{align}
E_H (f) \leq \liminf_{s \to \infty} E_H (f_s). \label{f-2}
\end{align}
It suffices to show that for any domain $ \Omega \subset M $ with an orthonormal basis $ \{e_A\}_{A = 1}^{2m} $ of $ HM \big|_{\Omega} $,
\begin{align}
\sum_{i, j, A} \int_\Omega h_{ij} (f) e_A f^i \: e_A f^j 
\leq 
\liminf_{s \to \infty} \sum_{i, j, A} \int_\Omega h_{ij} (f_s) e_A f^i_s \: e_A f^j_s.  \label{f-1}
\end{align}
For any $ \varepsilon > 0 $, since $f^i \in S^2_1 (\Omega)$ and $ f_s^i \to f^i $ strongly in $ L^2 (\Omega) $, there is a compact set $ K \subset \Omega $ such that
\begin{align*}
\sum_{i, j, A} \int_{\Omega \setminus K} h_{ij} (f) e_A f^i \: e_A f^j < \varepsilon 
\quad \mbox{and } \quad
f^i_s \rightrightarrows f^i  \quad \mbox{on } K,
\end{align*}
where ``$\rightrightarrows$'' means ``uniform convergence''.
The positivity of $ (h_{ij}) $ implies that
\begin{align*}
0 
& \leq \sum_{i, j, A} h_{ij} (f_s)  e_A (f^i_s - f^i) \: e_A (f^j_s - f^j)  \\
& = \sum_{i, j, A} h_{ij} (f_s)  e_A f^i_s \: e_A f^j_s 
- \sum_{i, j, A} h_{ij} (f_s)  e_A f^i \: e_A f^j 
- 2 \sum_{i, j, A} h_{ij} (f_s)  e_A f^i \: e_A (f^j_s - f^j) ,
\end{align*}
which yields that
\begin{align*}
\sum_{i, j, A} \int_K h_{ij} (f_s) e_A f^i_s \: & e_A f^j_s
 \geq  \sum_{i, j, A} \int_K h_{ij} (f_s) e_A f^i \: e_A f^j
+ 2 \sum_{i, j, A} \int_K h_{ij} (f_s) e_A f^i \: e_A (f^j_s - f^j) \\
& = \sum_{i, j, A} \int_K h_{ij} (f_s) e_A f^i \: e_A f^j
+
2 \sum_{i, j, A} \int_K (h_{ij} (f_s) - h_{ij} (f)) e_A f^i \: e_A (f^j_s - f^j) \\
& \quad + 2 \sum_{i, j, A} \int_K h_{ij} (f) e_A f^i \: e_A (f^j_s- f^j). \numberthis \label{f-3}
\end{align*}
For the first term of \eqref{f-3}, since $f_s^i \rightrightarrows f^i$ on $K$, then by mean value theorem, we have
\begin{align*}
\left| \sum_{i, j, A} \int_K (h_{ij} (f_s)- h_{ij} (f)) e_A f^i \: e_A f^j \right| \leq \sum_{i, j, k, A} \max_{B_D} \left| \frac{\partial h_{ij}}{\partial z^k} \right|  \int_K |f_s^k - f^k| \: |e_A f^i| \: |e_A f^j| \to 0,
\end{align*}
as $s \to \infty$, which implies that
\begin{align}
\lim_{s\to \infty} \sum_{i, j, A} \int_K h_{ij} (f_s) e_A f^i \: e_A f^j = \sum_{i, j, A} \int_K h_{ij} (f) e_A f^i \: e_A f^j. \label{f-4}
\end{align}
Similarly, since $ e_A f^j_s $ and $ e_A f^j $ are uniformly bounded in $ L^2 (K) $, then
\begin{align}
\lim_{s \to \infty} \sum_{i, j, A} \int_K (h_{ij} (f_s) - h_{ij} (f)) e_A f^i \: e_A (f^j_s - f^j) = 0. \label{f-5}
\end{align}
For the third term of \eqref{f-3}, define an operator $ T_A : S^2_1 (M) \to L^2 (K) $ by 
\begin{align*}
T_A (u) = e_A u \big|_K.
\end{align*}
$T_A$ is continuous due to the following calculation:
\begin{align*}
|| T_A (u) ||_{L^2 (K)}^2 = \int_K | e_A u |^2 \leq \int_M | \nabla_b u |^2 \leq || u ||_{S^2_1(M)}^2.
\end{align*}
Since any continuous operator between two Banach spaces preserves weak convergence, then $ e_A f^i_s \rightharpoonup e_A f^i $ weakly in $ L^2 (K) $ for any $ A $ and $ i $.
Hence 
\begin{align}
\lim_{s \to \infty} \sum_{i, j, A} \int_K h_{ij} (f) e_A f^i \: e_A (f^j_s- f^j) = 0 . \label{f-6}
\end{align} 
Using \eqref{f-4}, \eqref{f-5} and \eqref{f-6}, we find that
\begin{align*}
\sum_{i, j, A} \int_K h_{ij} (f) e_A f^i \: e_A f^j \leq \liminf_{s \to \infty} \sum_{i, j, A} \int_K h_{ij} (f_s) e_A f^i_s \: e_A f^j_s,
\end{align*}
which implies that
\begin{align*}
\sum_{i, j, A} \int_\Omega h_{ij} (f) e_A f^i \: e_A f^j
& \leq \sum_{i, j, A} \int_K h_{ij} (f) e_A f^i \: e_A f^j + \varepsilon \\
& \leq \liminf_{s \to \infty} \sum_{i, j, A} \int_K h_{ij} (f_s) e_A f^i_s \: e_A f^j_s + \varepsilon \\
& \leq \liminf_{s \to \infty} \sum_{i, j, A} \int_\Omega h_{ij} (f_s) e_A f^i_s \: e_A f^j_s + \varepsilon.
\end{align*}
By taking $\varepsilon \to 0$, we obtain \eqref{f-1} and thus $E_H (f) \leq \lambda$.

Obviously, $ E_H (f) \geq \lambda $ and then $ E_H (f) = \lambda $ which shows that $ f $ has the minimal horizontal energy in $ \mathcal{S} $ and satisfies
\begin{align*}
\Delta_b f^i + \Gamma^i_{jk} (f) \langle \nabla_b f^j, \nabla_b f^k \rangle = 0 ,
\end{align*}
in the distribution sense.
By applying Theorem 2 in \cite{jost1998subelliptic} to $ f $ on each coordinate neighborhood $ \Omega \Subset M $, we obtain the interior continuity of $ f $.
\end{proof}

\section*{Acknowledge}
The authors would like to thank the referees for their valuable comments.

\bibliographystyle{plain}

\bibliography{finalref}

\begin{thebibliography}{10}

\bibitem{Agrachev2015bishop}
A.~Agrachev and P.W.Y. Lee.
\newblock {Bishop and Laplacian Comparison Theorems on Three-Dimensional
  Contact Sub-Riemannian Manifolds with Symmetry}.
\newblock {\em J. Geom. Anal.}, 25(1):512--535, 2015.

\bibitem{barletta2001pseudoharmonic}
E.~Barletta, S.~Dragomir, and H.~Urakawa.
\newblock Pseudoharmonic {Maps} from {Nondegenerate} {CR} {Manifolds} to
  {Riemannian} {Manifolds}.
\newblock {\em Indiana Univ. Math. J.}, 50(2):719--746, 2001.

\bibitem{baudoin2017comparison}
F.~{Baudoin}, E.~{Grong}, K.~{Kuwada}, and A.~{Thalmaier}.
\newblock {Sub-Laplacian comparison theorems on totally geodesic Riemannian
  foliations}.
\newblock {\em ArXiv:1706.08489}, 2017.

\bibitem{boyer2008sasakian}
C.P. Boyer and K.~Galicki.
\newblock {\em ``Sasakian geometry"}.
\newblock Oxford University Press, Oxford, 2008.

\bibitem{chang2013existence}
S.C. Chang and T.H. Chang.
\newblock On the existence of pseudoharmonic maps from pseudohermitian
  manifolds into {R}iemannian manifolds with nonpositive sectional curvature.
\newblock {\em Asian J. Math.}, 17(1):1--16, 2013.

\bibitem{chang2018gradient}
S.C. {Chang}, T.J. {Kuo}, C.~{Lin}, and J.Z. {Tie}.
\newblock {CR Sub-Laplacian Comparison and Liouville-Type Theorem in a Complete
  Noncompact Sasakian Manifold}.
\newblock {\em {J. Geom. Anal.}}, 29(2):1676--1705, 2019.

\bibitem{chen2012exist}
Q.~Chen, J.~Jost, and H.B. Qiu.
\newblock {Existence and Liouville theorems for $V$-harmonic maps from complete
  manifolds}.
\newblock {\em Ann. Global Anal. Geom.}, 42(4):565--584, 2012.

\bibitem{cheng1980liouville}
S.Y. Cheng.
\newblock Liouville theorem for harmonic maps.
\newblock In {\em Proc. Sympos. Pure Math.}, volume~36, pages 147--151. Amer.
  Math. Soc., Providence, RI, 1980.

\bibitem{choi1982liouville}
H.I. Choi.
\newblock {On the Liouville theorem for harmonic maps}.
\newblock {\em Proc. Amer. Math. Soc.}, 85(1):91--94, 1982.

\bibitem{ding1991harmonic}
W.Y. Ding and Y.D. Wang.
\newblock {Harmonic maps of complete noncompact Riemannian manifolds}.
\newblock {\em Int. J. Math. Anal.}, 2(6):617--633, 1991.

\bibitem{do1992riemannian}
M.P. Do~Carmo.
\newblock {\em Riemannian geometry}.
\newblock Mathematics: Theory \& Applications. Birkh{\"a}user Basel, 1992.
\newblock Translated by Francis Flaherty.

\bibitem{dragomir2006differential}
S.~Dragomir and G.~Tomassini.
\newblock {\em {Differential geometry and analysis on CR manifolds}}.
\newblock Number 246 in Progress in Mathematics. Birkh{\"a}user Boston, Inc.,
  2006.

\bibitem{folland1974estimates}
G.B. Folland and E.M. Stein.
\newblock Estimates for the $\bar{\partial}_b$ complex and analysis on the
  {H}eisenberg group.
\newblock {\em Comm. Pure Appl. Math.}, 27:429--522, 1974.

\bibitem{greenleaf1985first}
A.~Greenleaf.
\newblock The first eigenvalue of a sublaplacian on a pseudohermitian manifold.
\newblock {\em Comm. Partial Differential Equations}, 10(2):191--217, 1985.

\bibitem{jost1998subelliptic}
J.~Jost and C.J. Xu.
\newblock Subelliptic harmonic maps.
\newblock {\em Trans. Amer. Math. Soc.}, 350(11):4633--4649, 1998.

\bibitem{lee2013bishop}
P.W.Y. {Lee} and C.~{Li}.
\newblock {Bishop and Laplacian comparison theorems on Sasakian manifolds}.
\newblock {\em Comm. Anal. Geom.}, 26(4):915--954, 2018.

\bibitem{li1998heat}
J.Y. Li and S.L. Wang.
\newblock The heat flows and harmonic maps from complete manifolds into regular
  balls.
\newblock {\em Bull. Aust. Math. Soc.}, 58(2):177--187, 1998.

\bibitem{Li1991heat}
P.~Li and L.F. Tam.
\newblock The heat equation and harmonic maps of complete manifolds.
\newblock {\em Invent. Math.}, 105(1):1--46, 1991.

\bibitem{nagel1985balls}
A.~Nagel, E.M. Stein, and S.~Wainger.
\newblock {Balls and metrics defined by vector fields I: Basic properties}.
\newblock {\em Acta Math.}, 155:103--147, 1985.

\bibitem{Ni1999hermitian}
L.~Ni.
\newblock {Hermitian harmonic maps from complete Hermitian manifolds to
  complete Riemannian manifolds}.
\newblock {\em Math. Z.}, 232(2):331--355, 1999.

\bibitem{ren2018pseudo}
Y.~{Ren} and G.~{Yang}.
\newblock {Pseudo-harmonic maps from closed pseudo-Hermitian manifolds to
  Riemannian manifolds with nonpositive sectional curvature}.
\newblock {\em Calc. Var. Partial Differential Equations}, 57(5):128, 2018.

\bibitem{Ren201447}
Y.B. Ren, G.L. Yang, and T.~Chong.
\newblock Liouville theorem for pseudoharmonic maps from {S}asakian manifolds.
\newblock {\em J. Geom. Phys.}, 81:47--61, 2014.

\bibitem{rothschild1976hypoelliptic}
L.P. Rothschild and E.M. Stein.
\newblock Hypoelliptic differential operators and nilpotent groups.
\newblock {\em Acta Math.}, 137:247--320, 1976.

\bibitem{strichartz1986sub}
R.S. Strichartz.
\newblock Sub-{R}iemannian {G}eometry.
\newblock {\em J. Differential Geom.}, 24(2):221--263, 1986.

\bibitem{tanaka1975differential}
N.~Tanaka.
\newblock {\em A differential geometric study on strongly pseudo-convex
  manifolds}, volume~9 of {\em Lectures in Mathematics, Department of
  Mathematics, Kyoto University}.
\newblock Kinokuniya Book-Store Co., 1975.

\bibitem{webster1978pseudo}
S.M. Webster.
\newblock {Pseudo-Hermitian structures on a real hypersurface}.
\newblock {\em J. Differential Geom}, 13(1):25--41, 1978.

\bibitem{xu1997higher}
C.J. Xu and C.~Zuily.
\newblock Higher interior regularity for quasilinear subelliptic systems.
\newblock {\em Calc. Var. Partial Differential Equations}, 5(4):323--343, 1997.

\end{thebibliography}

Tian Chong

\emph{School of Science, College of Arts and Sciences}

\emph{Shanghai Polytechnic University}

\emph{Shanghai, 201209, P. R. China}

chongtian@sspu.edu.cn

\vspace{12 pt}

Yuxin Dong

\emph{School of Mathematical Sciences}

\emph{Fudan University}

\emph{Shanghai, 200433, P. R. China}

yxdong@fudan.edu.cn

\vspace{12 pt}

Yibin Ren

\emph{College of Mathematics, Physics and Information Engineering}

\emph{Zhejiang Normal University}

\emph{Jinhua, 321004, Zhejiang, P.R. China}

allenryb@outlook.com

\vspace{12 pt}

Wei Zhang

\emph{School of Mathematics}

\emph{South China University of Technology}

\emph{Guangzhou, 510641, P.R. China}

sczhangw@scut.edu.cn

\end{document}